\documentclass[reqno]{amsart}
\pdfoutput=1 

\usepackage[all,hyperref]{bi-discrete} 
\usepackage[backend=biber,maxbibnames=5,maxalphanames=5,style=alphabetic,bibencoding=utf8,giveninits,url=false,isbn=false,eprint = true]{biblatex}
\renewbibmacro{in:}{}
\AtBeginBibliography{\small}

\usepackage{tikz}
\usetikzlibrary{decorations.pathreplacing,angles,quotes}
\usetikzlibrary{calc}

\usepackage[thinlines]{easytable}

\usepackage{booktabs}

\usepackage{pgfplots}
\pgfplotsset{compat=1.18} 

\usepackage{subcaption}
\usepackage{float}
\usetikzlibrary{patterns}

\newcommand{\ml}{\mathcal{L}}

\newcommand{\LK}{\mathcal{L}^1_K}

\newcommand{\metric}{d}

\newcommand{\dx}{\, \mathrm{d}x}

\begin{document}

\author[L.~Diening]{Lars Diening}
\author[V.~Lingert]{Viktoria Lingert}
\author[T.~Tscherpel]{Tabea Tscherpel}

\begin{abstract}
  We establish $L^p$- and $W^{1,p}$-stability of the $L^2$-projection onto mapped Lagrange finite elements on hybrid meshes consisting of triangles and convex quadrilaterals arising from adaptive mesh refinement.  
  If $K$ is the (tensor product) degree of polynomials of the discretisation, then we show, in particular, $W^{1,2}$-stability for all $K\geq 2$ for the Q-RG and Q-RB refinements. 
  This extends results by Ali, Funken, and Schmidt~\cite{AliFunkenSchmidt2022} which hold for the range $2 \leq K \leq 9$ for initial meshes consisting of parallelograms.  
  Our proof relies on an extension of the technique by Diening, Storn and Tscherpel in~\cite{DieningStornTscherpel2021} to general convex quadrilaterals.
\end{abstract}



\title[Sobolev stability of the $L^2$-projection on hybrid meshes]{Sobolev stability of the $L^2$-projection\\ on hybrid meshes}

\maketitle


\section{Introduction}
\label{sec:introduction}

Stability of the $L^2$-projection in norms other than $\norm{\cdot}_{L^2}$ is a key element in the numerical analysis of finite element methods.  For example, for the heat equation $W^{1,2}$-stability is equivalent to the inf-sup stability and quasi-optimality of Galerkin methods (see~\cite{TantardiniVeeser2016}).  For nonlinear evolution problems stability properties provide an important ingredient in the proof of error estimates that do not rely on a coupling of time and space discretisation parameters~\cite{BreitDieningStornEtAl2021}.  Furthermore, decay properties of the $L^2$-projection are instrumental for results on $hp$-preconditioning~\cite{GrahamSpence2025}.

Let $\Pi$ denote the $L^2$-projection onto the Lagrange finite element space $V_{h}$ of a certain polynomial degree $K \geq 1$. 
We aim for stability in $L^p$ and $W^{1,p}$ for $p \in [1,\infty]$, i.e., we want to show that 
\begin{align*}
	  \norm{\Pi u}_{p} \lesssim \norm{u}_{p} \qquad \text{ and } \qquad 
  \norm{\nabla \Pi u}_{p} \lesssim \norm{\nabla u}_{p},
\end{align*}
holds for any $u \in L^p(\Omega)$ or $u \in W^{1,p}(\Omega)$ uniformly in the mesh size, respectively. 

As proved in~\cite{BrambleXu1991} for quasi-uniform meshes stability in $W^{1,2}$ follows from inverse estimates and by using a Sobolev stable operator like the Scott--Zhang operator~\cite{ScottZhang1990} or the Clément operator~\cite{Clement1975}.  
However, both for non-quasiuniform meshes and for $p \neq 2$ establishing stability results is a highly challenging task. 
For the analysis of adaptive finite element methods it is important that Sobolev stability holds for graded meshes arising from adaptive mesh refinement, uniformly in the refinement level. 
Furthermore, stability estimates for $p \neq 2$ are crucial in the error analysis of nonlinear evolution equations~\cite{BreitDieningStornEtAl2021}.

In dimension $d = 1$ it is known that Sobolev stability cannot hold without conditions on the mesh grading~\cite{CrouzeixThomee1987,BankYserentant2014}. 
In fact,  in~\cite{BankYserentant2014} a counterexample to $W^{1,2}$-stability is presented for the lowest order case when neighbouring intervals differ sufficiently in size.  
In higher dimensions analogous counterexamples are unknown, but all available stability results are based on certain grading assumptions. 
Specifically, this means that the local mesh sizes of neighbouring elements are not allowed to differ too strongly.

For \emph{simplicial meshes} in dimension $d \geq 2$ numerous stability results of the $L^2$-projection are available.  
Beginning with the work on mildly graded triangulations in~\cite{CrouzeixThomee1987,Boman2006,ErikssonJohnson1995}, stability results were subsequently established for low-order spaces on highly graded meshes, such as those generated by adaptive mesh refinement~\cite{Steinbach2001, Steinbach2002,BramblePasciakSteinbach2002,Carstensen2002,Carstensen2004}. 
More recently, stability was established for a bounded range of polynomial degrees and on meshes exhibiting realistic grading~\cite{BankYserentant2014,GaspozHeineSiebert2016,GaspozHeineSiebert2019} for dimension $d = 2$ and in~\cite{BankYserentant2014} for dimension $d = 3$. 
The most recent results in~\cite{DieningStornTscherpel2021,DieningStornTscherpel2026} build on the techniques developed in~\cite{BankYserentant2014} and establish mesh grading results for BisecMT~\cite{Maubach1995,Traxler1997}, the generalisation of the newest vertex bisection, in any space dimension $d \in \mathbb{N}$. 
Taken together, these works establish $W^{1,2}$-stability for triangulations generated by BisecMT in any dimension $d \leq 6$ for any polynomial degree $K \geq 1$.  
More generally, $W^{1,p}$- and $L^p$-stability are established in~\cite{DieningStornTscherpel2021} under certain grading assumptions by means of an argument utilising maximal operators.

In addition to simplicial meshes, quadrilateral meshes are frequently employed for finite element discretisation~\cite{ArnoldBoffiFalkEtAl2001,ArnoldBoffiFalk2002}, see also~\cite[Ch.~4, Sec.~4.3]{Ciarlet2002}.  
In contrast to simplicial meshes, structured quadrilateral meshes admit a tensor product structure enabling highly efficient implementations. 
However, adaptive mesh refinement on purely quadrilateral meshes typically introduces hanging vertices. 
To address this, \emph{hybrid meshes} consisting of simplices and quadrilaterals offer a flexible framework that accommodates adaptive refinement while retaining mesh conformity. 
Such refinement schemes have been presented in~\cite{BankShermanWeiser1983,Kobbelt1996,ZhaoMaoShi2010}. 
In particular, quadrilaterals are well-suited for anisotropic meshes~\cite{Apel1999} which are often employed in the presence of edge singularities, see, e.g.,~\cite{FaustmannMarcatiMelenkEtAl2023}. 
This work, however, focuses on isotropic shape-regular hybrid meshes. 

To date, the only stability results for the $L^2$-projection onto Lagrange elements on hybrid meshes are due to Ali, Funken, and Schmidt~\cite{AliFunkenSchmidt2022}. 
In their work, the authors consider two-dimensional hybrid meshes generated via adaptive Q-RG~\cite{BankShermanWeiser1983} and Q-RB~\cite{Kobbelt1996} refinement starting from initial meshes composed exclusively of parallelograms. 
The resulting hybrid meshes consist of triangles and trapezoids.  
For these meshes, they obtain mesh grading results by applying the approach from~\cite{Carstensen2004, GaspozHeineSiebert2016}. 
Furthermore, by employing the technique from~\cite{BramblePasciakSteinbach2002,Carstensen2002} and numerically solving certain local generalised eigenvalue problems, they show $W^{1,2}$-stability for polynomial degree $K \in \{2, \ldots, 9\}$.   
\medskip 

This work extends previous stability results for the $L^2$-projection onto Lagrange finite element spaces on two-dimensional meshes containing simplices and quadrilaterals in the following directions: 
\noindent%
\begin{itemize}[leftmargin=5mm]
\item 
Our results apply to general hybrid meshes consisting of triangles and general convex quadrilaterals that satisfy certain grading and shape-regularity properties. 
This includes, in particular, meshes generated by the Q-RG and Q-RB refinement scheme starting from an initial mesh of general quadrilaterals; see Remark~\ref{rem:general-grading} below. 
This generalises the results in~\cite{AliFunkenSchmidt2022} for initial meshes consisting exclusively of parallelograms to those consisting of general quadrilaterals. 
\item 
 	Our stability results hold under general mesh grading assumptions and can be readily applied to other refinement schemes that yield hybrid meshes satisfying grading and shape-regularity properties.

\item 
 We extend the valid range of polynomial degrees from $K \in \{2, \ldots, 9\}$ to all $K \geq 2$. 
 	Since the decay properties of~$\Pi$ improve for higher polynomial degrees, this extension provides a foundation for future investigation of $hp$-methods. 		

\item  
We establish $W^{1,p}$- and $L^p$-stability by following the framework in~\cite[Sec.~4]{DieningStornTscherpel2021}, which relies solely on weighted $L^2$-estimates and the stability of a quasi-interpolation operator, such as the Scott--Zhang operator.  
Consequently, any future improvements in the weighted $L^2$-estimates will immediately transfer into enhanced $W^{1,p}$ and $L^p$-stability results. 
\end{itemize}

\emph{Strategy and outline.}  
In Section~\ref{sec:mesh-funct-spac}, we introduce finite element spaces on general hybrid meshes and review some adaptive mesh refinement schemes. 
Furthermore, we collect available grading and shape-regularity results.  
To establish the stability results in Section~\ref{sec:stab}, we follow the general strategy in~\cite{DieningStornTscherpel2021}, which improves upon the approach in~\cite{BankYserentant2014}.  
The core idea is to use a local approximation operator, see Section~\ref{sec:appr-oper}, to establish decay estimates.  
A (local) decomposition operator, see Sections~\ref{sec:decomp-oper} and~\ref{sec:global-decomposition}, is key to verifying estimates for the approximation operator.  
To extend the results from simplicial meshes to general quadrilaterals, we utilise the tensor-product structure of the local spaces, but the fact that the transformations are in general non-affine poses extra challenges.
 While some of these difficulties can be avoided by restricting the initial meshes to parallelograms as in~\cite{AliFunkenSchmidt2022}, 
we address the more general case of arbitrary convex quadrilaterals. 
The remaining steps -- deducing a decay estimate in Section~\ref{sec:decay-estimates}, weighted $L^2$-stability and finally $L^p$- and $W^{1,p}$-stability in Section~\ref{sec:stab-res} -- proceed analogously to the simplicial case.  
They rely on results analogous to those in~\cite[Sec.~4]{DieningStornTscherpel2021}. 

\emph{Extensions and limitations}. 
In general dimension $d \geq 3$, the only available grading results  are those in~\cite{DieningStornTscherpel2026} for simplicial meshes and BisecMT refinement (the generalised newest vertex bisection). 
 The proof of this result is considerably more challenging than its two-dimensional counterparts, because a direct enumeration and investigation of worst-case scenarios is not possible.  
 Furthermore, refinement schemes and corresponding grading results for hybrid meshes are currently unavailable in higher dimensions. 
 Simply relying on the tensor-product structure in combination with the techniques from~\cite{DieningStornTscherpel2021} fails to yield suitable estimates for dimensions $d\geq 3$, see Remark~\ref{rmk:general-dim} below.  For this reason, establishing stability results on hybrid meshes in higher dimensions necessitates alternative strategies.   


\section{Meshes and function spaces}
\label{sec:mesh-funct-spac}

In this work we consider two-dimensional hybrid meshes consisting of triangles and (strictly) convex quadrilaterals. 
In this section we introduce the precise setup for the meshes and the finite dimensional function spaces that we use. 

\subsection{Hybrid meshes}\label{sec:mesh}

Let $\Omega \subset \mathbb{R}^2$ be a bounded, polyhedral domain. 
We write $a \lesssim b$, if there is a constant~$c>0$ only depending on specified parameters such that~$a \leq c b$. 
We write $a \eqsim b$, if $a \lesssim b$ and $b \lesssim a$.

We decompose~$\Omega$ into a mesh~$\mathcal{T}$ consisting of (closed) triangles and (strictly) convex quadrilaterals, where \emph{element} refers to both triangles and quadrilaterals. 
The latter are called quadrilaterals of type (1) in~\cite[Ch.~4, Sec.~4.3]{Ciarlet2002}.

 We suppose that the mesh~$\mathcal{T}$ is regular, in the following sense: 
\begin{enumerate}
\item The mesh covers~$\overline{\Omega}$, i.e., $\overline{\Omega}=\bigcup_{T\in\mathcal{T}} T$. 
\item If $T_1,T_2 \in \mathcal{T}$ overlap, then $T_1 \cap T_2$  is a joint edge or vertex in~$\mathcal{T}$. 
In particular, there are no hanging vertices.
\end{enumerate}
Let $\mathcal{V} = \mathcal{V}(\mathcal{T})$ denote the set of vertices of $\mathcal{T}$ and let $\omega_i \coloneqq \{T \in \mathcal{T} \colon i \in T\}$ denote the vertex patch of a vertex $i \in \mathcal{V}$. 
	By $\Omega_i \coloneqq (\cup_{T \in \omega_i} T)^\circ $ we denote the vertex patch domain.  

In the following we refer to a hybrid mesh $\mathcal{T}$ as 
\begin{itemize}
\item \emph{simplicial}, if $\mathcal{T}$ contains only triangles, 
\item \emph{quadrilateral}, if $\mathcal{T}$ contains only (strictly) convex quadrilaterals,
\item \emph{P-quadrilateral}, if $\mathcal{T}$ contains only parallelograms, 
\item \emph{SP-hybrid}, if $\mathcal{T}$ contains only triangles and parallelograms, and 
\item \emph{general hybrid}, if $\mathcal{T}$ contains triangles and general convex quadrilaterals.
\end{itemize}

We also use the concept of shape-regularity which measures the extent of deformation of elements. 
We employ the standard approach which applies both for triangles and quadrilaterals, see~\cite[Ch.~4, \S~4.3]{Ciarlet2002}.  

If $T \in \mathcal{T}$ is a triangle, then we define $\widehat{T}$ as the equilateral triangle with vertices $(0,0)^\top, (1,0)^\top, (\frac 12,\frac{\sqrt{3}}{2})^\top$ and diameter $\diameter(\widehat{T}) = 1$. 
If $T$ is a quadrilateral, then~$\widehat{T}$ 
denotes the unit square~$[0,1]^2$. 
For triangles there exists an affine linear isomorphism $B_T\colon \widehat{T} \to T$ and for quadrilaterals there exists an affine bilinear isomorphism~$B_T\colon \widehat{T} \to T$, see Figure~\ref{fig:transformation}. 
By $J_T$ we denote the Jacobian determinant of~$B_T$. 
Since $T$ is strictly convex, one can choose~$B_T$ such that~$J_T$ is positive on~$\widehat{T}$. 

\begin{figure}[H]
  \begin{tikzpicture}
    \begin{scope}[shift={(0,0)}]
      \draw (0,0) -- (1,0) -- (1/2,0.866) -- cycle;
      \node at (0.5,0.25) {\scriptsize $\widehat{T}$};
    \end{scope}
    \begin{scope}[shift={(1,0)}]
      \draw[->] (0.2,0.5) -- (0.7,0.5);
      \node at (0.45,0.65) {\scriptsize $B_T$};
    \end{scope}
    \begin{scope}[shift={(2,0)}]
      \draw (0,0) -- (1.2,0.5) -- (0.2,1) -- cycle;
      \node at (0.4,0.45) {\scriptsize $T$};
    \end{scope}
    \begin{scope}[shift={(5,0)}]
      \draw (0,0) rectangle (1,1);
      \node at (0.5,0.5) {\scriptsize $\widehat{T}$};
    \end{scope}
    \begin{scope}[shift={(6.2,0)}]
      \draw[->] (0.2,0.5) -- (0.7,0.5);
      \node at (0.45,0.65) {\scriptsize $B_T$};
    \end{scope}
    \begin{scope}[shift={(7,0)}]
      \draw (0,0) -- (1.2,-0.2) -- (1.8,1.1) -- (0.6,0.8) -- cycle;
      \node at (0.8,0.40) {\scriptsize $T$};
    \end{scope}
  \end{tikzpicture}
  \caption{Transformation of a triangle (left) and a convex quadrilateral (right).}
  \label{fig:transformation}
\end{figure}
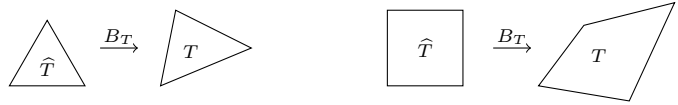

The \emph{shape-regularity constant} of an element $T \in \mathcal{T}$ is defined as
\begin{align}\label{def:chi-T}
  \chi(T)  \coloneqq  \max_{\widehat{x}\in \widehat{T}} \norm{\nabla B_T(\widehat{x})} \max_{x \in T} \norm{\nabla B_T^{-1}(x)},
\end{align}
where $\norm{\cdot}$ is the spectral norm which is induced by the Euclidean vector norm.  Note that $\chi(T) \geq 1$, and that~$\chi(T)$ increases with the deformation of~$T$.

The mapping $B_T$ is not unique, since the vertices can be rotated. 
However, due to the symmetry of~$\widehat{T}$, the definition of~$\chi(T)$ in~\eqref{def:chi-T} is independent of the specific choice of~$B_T$.
This independence was also observed by Knupp~\cite{Knupp2001}, where he discusses several mesh-quality metrics. 
Moreover, by using the equilateral triangle and the unit square as reference elements, we obtain that~$\chi(\widehat{T})=1$.

By standard arguments, one can show that for $T\in \mathcal{T}$ one has
\begin{alignat}{2}
  \label{eq:quotient-diameter}
  \frac{1}{\max_T \norm{\nabla B_T^{-1}}} 
  &\leq \frac{\diameter(T)}{\diameter(\widehat{T})} 
  &&\leq \max_T \norm{\nabla B_T},
  \\
  \label{eq:quotient-volume}
  \frac{1}{\max_T \abs{J_T^{-1}}} &\leq \quad\;\frac{\abs{T}}{\abs{\widehat{T}}}
   &&
     \leq
   \max_T \abs{J_T}
\end{alignat}
Thus, the estimates~$\abs{J_T}^{\frac 12} \leq \norm{\nabla B_T}$ and~$\abs{J_T^{-1}}^{\frac 12} \leq \norm{\nabla B_T^{-1}}$ imply that
\begin{align}
  \label{eq:diam-vs-volume}
  \frac{1}{\chi(T)} \frac{\diameter(\widehat{T})}{\abs{\widehat{T}}^{\frac 12}} \leq \frac{\diameter(T)}{\abs{T}^{\frac 12}} \leq \chi(T) \frac{\diameter(\widehat{T})}{\abs{\widehat{T}}^{\frac 12}} .
\end{align}
We define $\chi(\mathcal{T}) \coloneqq \max_{T \in \mathcal{T}} \chi(T)$  as the shape-regularity constant of a hybrid mesh~$\mathcal{T}$.
If $\chi(\mathcal{T}_j)$ is bounded for a family of meshes $(\mathcal{T}_j)_{j \in \mathbb{N}}$, then, in particular, we have  
\begin{align}
  \label{eq:shape-regular}
  \diameter(T) \eqsim \abs{T}^{\frac 12} \qquad \text{for any $T \in \mathcal{T}_j$,} 
  \;\;
  \text{for any $j \in \mathbb{N}$,}
\end{align}
i.e., uniformly for the family of meshes. 

\begin{remark}
  \label{rem:shape-regularity-triangles}
  For triangles, one additionally can show that 
  \begin{gather}    
    \label{eq:quotient-diameter2}
    \frac{1}{\norm{\nabla B_T^{-1}}} 
    \leq \frac{\rho(T)}{\rho(\widehat{T})} 
    \leq \norm{\nabla B_T},
    \\
    \label{eq:quotient-triangle}
    \norm{\nabla B_T} \leq \frac{\diameter(T)}{\rho(\widehat{T})}
    \qquad \text{and} \qquad
    \norm{\nabla B_T^{-1}} \leq \frac{\diameter(\widehat{T})}{\rho(T)},
\end{gather} 
 cf.~\cite[Ch.~11.1]{ErnGuermond2021}. 
As a consequence, in combination with~\eqref{eq:quotient-diameter} one has 
\begin{gather}
    \frac{\rho(\widehat{T})}{\diameter(\widehat{T})}\, \frac{\diameter(T)}{\rho(T)} \leq \chi(T) \leq \frac{\diameter(\widehat{T})}{\rho(\widehat{T})}\, \frac{\diameter(T)}{\rho(T)}. 
  \end{gather}
  This shows that, for triangles, $\chi(T)$ is equivalent to $\diameter(T)/\rho(T)$, which is usually introduced as the shape-regularity constant.
\end{remark} 

\subsection{Grading}
\label{sec:grading}

Let us collect the notations for distance, graded weight functions and mesh functions as  in~\cite{DieningStornTscherpel2021} adapted to the setting of hybrid meshes.

We start by defining a metric~$\metric$ on~$\mathcal{T}$.  For $T \in \mathcal{T}$ we define $\metric(T,T) \coloneqq 0$.  Moreover, if $T,T' \in \mathcal{T}$ with $T \neq T'$ are neighbours, i.e., $T \cap T' \not= \emptyset$, then we set~$\metric(T,T') \coloneqq1$.  For all other $T,T' \in \mathcal{T}$ we define $\metric(T,T')$ as the length of the shortest chain $T_0,\dots, T_m \in \mathcal{T}$ with $T_0 =T$ and $T_m=T'$ and $\metric(T_j,T_{j+1})=1$.  For subsets $L, L' \subset \mathcal{T}$ we define $\metric(L,L') \coloneqq \min \set{\metric(T,T')\colon T \in L, \, T' \in L'}$.

By $\mathcal{L}^0_0(\mathcal{T})$ we denote the set of $L^1(\Omega)$ functions that are piecewise constant on~$\mathcal{T}$. A \emph{weight} $\rho$ with \emph{grading} $\gamma_\rho\geq 1$ is a positive function~$\rho \in \mathcal{L}^0_0(\mathcal{T})$ with
\begin{align}
  \label{eq:grading}
  \rho|_T \leq \gamma_\rho \rho|_{T'} \qquad \text{for all $T,T'\in\mathcal{T}$ with $\metric(T,T')=1$}.
\end{align}
Note that this is equivalent to
\begin{align}
  \label{eq:grading2}
  \rho|_T \leq \gamma_\rho^{\metric(T,T')} \rho|_{T'} \qquad \text{for all $T,T' \in \mathcal{T}$}.
\end{align}

We call~$h\in \mathcal{L}^0_0(\mathcal{T})$ a (regularised) \emph{mesh size function} of~$\mathcal{T}$, if
\begin{align}
  \label{eq:hT}
  h|_T &\eqsim \diameter(T) \qquad \text{ for all } T \in \mathcal{T},
\end{align}
where the implicit constant does not depend on~$T \in \mathcal{T}$.

We say that the mesh~$\mathcal{T}$ has \emph{grading}~$\gamma_{\mathcal{T}} \geq 1$, if there exists a mesh size function~$h \in \mathcal{L}^0_0(\mathcal{T})$ with grading~$\gamma_{\mathcal{T}}$, i.e.,
\begin{align}\label{def:grading}
  \frac{h|_{T}}{h|_{T'}} \leq \gamma_{\mathcal{T}}  \quad \text{ for all } T,T' \in \mathcal{T} \text{ with } T \cap T' \neq \emptyset.
\end{align}
For a family of adaptively refined meshes, later a uniform maximal grading constant $\gamma_{\mathcal{T}} \leq \gamma$ will be important.

\subsection{Function spaces}\label{sec:fct-space}

By $L^p(\Omega)$ and $W^{1,p}(\Omega)$ we denote the usual Lebesgue and Sobolev spaces with norms $\norm{\cdot}_p$ and $\norm{\cdot}_{1,p}$, respectively, for $p \in [1,\infty]$. 
For~$K \in \setN$ by $\mathcal{P}_K(U)$ we denote the set of polynomials on a set~$U$ of total degree~$K$ (or less). 
Moreover, we denote the polynomials of maximal degree $K$ on~$[0,1]^{2}$ by
\begin{equation}\label{def:QK}
\begin{aligned}
  \mathcal{Q}_K([0,1]^2) 
  & \coloneqq \mathcal{P}_K([0,1]) \otimes \mathcal{P}_K([0,1]) \\
  & \coloneqq \linearspan \set{ (x_1,x_2)^\top \mapsto p_1(x_1)p_2(x_2)\colon p_1,p_2 \in \mathcal{P}_K([0,1])}.
\end{aligned}
\end{equation}

Let us define the local spaces on triangles and on (strictly) convex quadrilaterals. 
As in Section~\ref{sec:mesh} let $B_T \colon \widehat{T} \to T$ denote the isomorphism from the reference element~$\widehat{T}$ to the element~$T$. 
For a triangle $T \in \mathcal{T}$ we define
\begin{align}\label{def:LK-tria}
  \mathcal{L}_K(T) &\coloneqq \set{ p \circ B_T^{-1}\colon p \in \mathcal{P}_K(\widehat{T})} = \mathcal{P}_K(T), 
\end{align}
and for a quadrilateral $T \in \mathcal{T}$ we define
\begin{align}\label{def:LK-quad}
  \mathcal{L}_K(T) &\coloneqq \set{ p \circ B_T^{-1}\colon p \in \mathcal{Q}_K([0,1]^2)},
\end{align}
which is in general not a space of polynomials. 
Note that the definition of~$\mathcal{L}_K(T)$ is independent of the choice of~$B_T$.

We define the global space of discrete functions over~$\Omega$ as
\begin{align}\label{def:L1K}
  \mathcal{L}^1_K(\mathcal{T}) &\coloneqq \set{ v \in W^{1,\infty}(\Omega) \colon v|_T \in \mathcal{L}_K(T) \text{ for all } T \in \mathcal{T}}.
\end{align}
This is the special case of isoparametric elements for bilinear transformations. 
We refer to $\mathcal{L}^1_K(\mathcal{T})$ as (bilinearly) mapped Lagrange elements. 
These spaces have approximation order~$K+1$ in $L^2$ and order $K$ in $W^{1,2}$ in terms of~$h$, see~\cite{ArnoldBoffiFalk2002}. 

\subsubsection*{Function space decomposition}
\label{sec:space-decomposition}

Let us denote by~$(\phi_i)_{i\in \mathcal{V}}$ the nodal basis of~$\mathcal{L}^1_1(\mathcal{T})$, which forms a partition of unity. 
Note that on triangles the $\phi_i$ are the barycentric coordinates and thus affine linear. 
On quadrilaterals the $\phi_i$ are the pullback of the tensor products of the one-dimensional barycentric coordinates of the standard square. 
If~$T$ is a parallelogram, then $B_T$ is affine and $\phi_i$ is affine bilinear.  
In the other cases $\phi_i$ is a non-singular rational function.

Using the partition of unity $(\phi_i)_{i\in \mathcal{V}}$ we have the following decomposition
\begin{align}
  \label{eq:space-decomposition}
  \mathcal{L}^1_K(\mathcal{T})=\sum_{i\in\mathcal{V}}\phi_{i}\mathcal{L}_{K-1}^{1}(\omega_{i}).
\end{align}

\subsection{Adaptive mesh refinement}
\label{sec:adapt-mesh-refin}

$W^{1,2}$-stability of the $L^2$-projection mapping to Lagrange finite element spaces requires mesh conditions, specifically conditions on the mesh grading as in~\eqref{def:grading}.  
In dimension $d = 1$ , there are counterexamples~\cite[Sec.~7]{BankYserentant2014} showing that for a mesh with sufficiently strong grading, this stability cannot hold. 
In particular, the authors construct a sequence of one-dimensional meshes with large grading such that~$W^{1,2}$-stability fails for~$\mathcal{L}^1_1(\mathcal{T})$.   
A finer analysis allows to quantify the grading threshold of this example as
\begin{align}\label{est:gamma-threshold}
\gamma \geq 1/(1+\sqrt{3}-\sqrt{3+2\sqrt{3}}) \approx 5.274510564406.
\end{align}

In~\cite{AliFunkenSchmidt2022} adaptive quadrilateral meshes refined by local red, green and blue refinements are analysed. 
The local refinements of $[0,1]^2$ are depicted in the first row of Figure~\ref{fig:QRG}, and for general convex quadrilaterals in the second row. 
In the latter case the bilinear mapping~$B_T \colon [0,1]^2 \to T$ is used to determine the  new vertices in~$T$. 
Then, the corresponding vertices are connected by straight lines in~$T$, see Figure~\ref{fig:QRG} second line. 
Note that straight lines in $[0,1]^2$ are mapped to straight lines in $T$ only if they are axis-parallel. 
Hence, mapping the new edges in $[0,1]^2$ does not lead to quadrilaterals with straight edges after mapping them.  

The global adaptive refinement routine uses red refinements, and temporary green refinements (Q-RG) or temporary blue refinements (Q-RB) to remove hanging nodes. 
The Q-RG and Q-RB refinement routines were originally introduced in~\cite{BankShermanWeiser1983} and in~\cite{Kobbelt1996}, respectively. 
Grading properties are established in \cite{AliFunkenSchmidt2022} and the proof uses a regularised mesh size function introduced in~\cite{Carstensen2002}.

\begin{figure}[t!]
	\begin{tikzpicture}[scale=1.1]
		\begin{scope}[scale=2,shift={(1.2,0)}]
			\draw[fill=red!30] (0,0) rectangle (1,1);
			\fill (0,0) circle (1pt);
			\fill (1,0) circle (1pt);
			\fill (0,1) circle (1pt);
			\fill (1,1) circle (1pt);
			\draw (1/2,0) -- (1/2,1);
			\draw (0,1/2) -- (1,1/2);
			\fill (1/2,0) circle (1pt);
			\fill (1/2,1) circle (1pt);
			\fill (0,1/2) circle (1pt);
			\fill (1,1/2) circle (1pt);
			\fill (1/2,1/2) circle (1pt);
		\end{scope}
		\begin{scope}[scale=2,shift={(2.4,0)}]
			\draw[fill=rltgreen!30] (0,0) rectangle (1,1);
			\fill (0,0) circle (1pt);
			\fill (1,0) circle (1pt);
			\fill (0,1) circle (1pt);
			\fill (1,1) circle (1pt);
			\draw (0,1) -- (1/2,0) -- (1,1);
			\fill (1/2,0) circle (1pt);
		\end{scope}
		\begin{scope}[scale=2,shift={(3.6,0)}]
			\draw[fill=rltgreen!30] (0,0) rectangle (1,1);
			\fill (0,0) circle (1pt);
			\fill (1,0) circle (1pt);
			\fill (0,1) circle (1pt);
			\fill (1,1) circle (1pt);
			\draw (0,1) -- (1/2,0) -- (1,1/2) -- cycle;
			\fill (1/2,0) circle (1pt);
			\fill (1,1/2) circle (1pt);
		\end{scope}
		\begin{scope}[scale=2,shift={(4.8,0)}]
			\draw[fill=rltgreen!30] (0,0) rectangle (1,1);
			\fill (0,0) circle (1pt);
			\fill (1,0) circle (1pt);
			\fill (0,1) circle (1pt);
			\fill (1,1) circle (1pt);
			\draw (1/2,0) -- (1/2,1);
			\fill (1/2,0) circle (1pt);
			\fill (1/2,1) circle (1pt);
		\end{scope}
		\begin{scope}[scale=2,shift={(6.0,0)}]
			\draw[fill=cyan!30] (0,0) rectangle (1,1);
			\fill (0,0) circle (1pt);
			\fill (1,0) circle (1pt);
			\fill (0,1) circle (1pt);
			\fill (1,1) circle (1pt);
			\fill (1/2,0) circle (1pt);
			\fill (1,1/2) circle (1pt);
			\fill (1/2,1/2) circle (1pt);
			\draw (0,1) -- (1/2,1/2);
			\draw (1,1/2) -- (1/2,1/2);
			\draw (1/2,0) -- (1/2,1/2);
		\end{scope}
	\end{tikzpicture}

  \bigskip%
	\begin{tikzpicture}[scale=1.1]
		
		\def\setcoords{
			\coordinate (A) at (0.1,0);
			\coordinate (B) at (0.9,0.1);
			\coordinate (C) at (1,1.2);
			\coordinate (D) at (0,1);
			
			\coordinate (AB) at ($(A)!0.5!(B)$);
			\coordinate (BC) at ($(B)!0.5!(C)$);
			\coordinate (CD) at ($(C)!0.5!(D)$);
			\coordinate (DA) at ($(D)!0.5!(A)$);
			
			\coordinate (M)  at ($(AB)!0.5!(CD)$);
		}

		\begin{scope}[scale=2,shift={(1.2,0)}]
			\setcoords
			\draw[fill=red!30] (A) -- (B) -- (C) -- (D) -- cycle;
			\draw (AB) -- (CD);
			\draw (DA) -- (BC);
			\foreach \p in {A, B, C, D, AB, BC, CD, DA, M} \fill (\p) circle (1pt);
		\end{scope}
		
		\begin{scope}[scale=2,shift={(2.4,0)}]
			\setcoords
			\draw[fill=rltgreen!30] (A) -- (B) -- (C) -- (D) -- cycle;
			\draw (D) -- (AB) -- (C);
			\foreach \p in {A, B, C, D, AB} \fill (\p) circle (1pt);
		\end{scope}
		
		\begin{scope}[scale=2,shift={(3.6,0)}]
			\setcoords
			\draw[fill=rltgreen!30] (A) -- (B) -- (C) -- (D) -- cycle;
			\draw (D) -- (AB) -- (BC) -- cycle;
			\foreach \p in {A, B, C, D, AB, BC} \fill (\p) circle (1pt);
		\end{scope}
		
		\begin{scope}[scale=2,shift={(4.8,0)}]
			\setcoords
			\draw[fill=rltgreen!30] (A) -- (B) -- (C) -- (D) -- cycle;
			\draw (AB) -- (CD);
			\foreach \p in {A, B, C, D, AB, CD} \fill (\p) circle (1pt);
		\end{scope}
		
		\begin{scope}[scale=2,shift={(6.0,0)}]
			\setcoords
			\draw[fill=cyan!30] (A) -- (B) -- (C) -- (D) -- cycle;
			\draw (D) -- (M);
			\draw (BC) -- (M);
			\draw (AB) -- (M);
			\foreach \p in {A, B, C, D, AB, BC, M} \fill (\p) circle (1pt);
		\end{scope}
		
	\end{tikzpicture}
	\caption{Local refinement of $[0,1]^2$ (first line) and of a general convex quadrilateral (second line) as used in Q-RG and Q-RB refinement, see~\cite{AliFunkenSchmidt2022} and~\cite{BankShermanWeiser1983,Kobbelt1996}}
	\label{fig:QRG}
\end{figure}
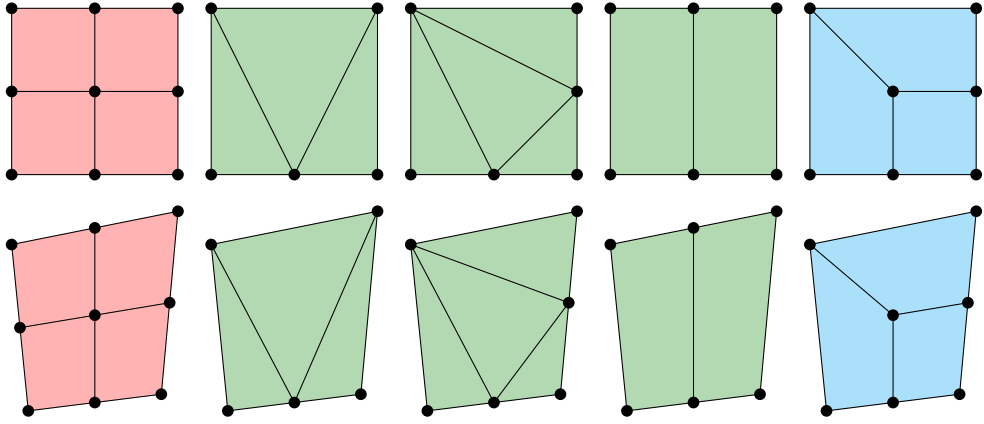

\begin{lemma} \label{lem:refinement-properties}
  The local element shapes generated by red, green, and blue refinement, as well as the global mesh types generated by the Q-RG and Q-RB strategies, strictly satisfy the closure properties summarised in Table~\ref{tab:refinement-rules}.
\end{lemma}
\begin{proof}
  We only verify in detail that the blue refinement (see Figure~\ref{fig:QRG}) of a convex quadrilateral generates quadrilaterals that are convex, since all other statements are simple to check. 
  It suffices to prove convexity of the blue quadrilateral in the bottom left corner in Figure~\ref{fig:QRG}. 
  For this we observe that the new (axis-parallel) line~$L$ in the middle of the third green refinement in Figure~\ref{fig:QRG} is mapped by the bilinear map $B_T\colon  [0,1]^2\to T$ to a straight line $B_T(L)$. 
  Then, the blue quadrilateral can be generated from the green quadrilateral by moving the point $B_T((1/2,1)^\top)$ to the centre of the line~$B_T(L)$. 
  The convexity is preserved by this operation.
\end{proof}

\begin{table}[ht!]
  \centering
  \renewcommand{\arraystretch}{1.2}
  \begin{tabular}{lll}
    \toprule
    \textbf{operation} & \textbf{initial element} & \textbf{resulting elements} \\
    \midrule
    \multicolumn{3}{c}{\textit{local element refinement}} \\
    \midrule
    red & parallelogram & parallelograms \\
    red & convex quadrilateral & convex quadrilaterals \\
    green & parallelogram & triangles and parallelograms \\
    green & convex quadrilateral & triangles and convex quadrilaterals \\
    blue & convex quadrilateral & convex quadrilaterals \\
    \midrule
    \multicolumn{3}{c}{\textit{global mesh refinement}} \\
    \midrule
    Q-RG & P-quadrilateral $\mathcal{T}_0$ & SP-hybrid meshes \\
    Q-RB & P-quadrilateral $\mathcal{T}_0$ & general hybrid meshes \\
    Q-RG / Q-RB & quadrilateral $\mathcal{T}_0$ & general hybrid meshes \\
    \bottomrule
  \end{tabular}
  \caption{Closure properties of individual elements and full meshes under different refinement strategies.}
  \label{tab:refinement-rules}
\end{table}

  \begin{lemma}
    \label{lem:shape-regular-Q-RG-Q-RB}
  Let $\mathcal{T}_0$ be a quadrilateral mesh. Then, the families of meshes generated by the adaptive refinement routines Q-RG or Q-RB are uniformly shape-regular. More precisely, for any such mesh~$\mathcal{T}$ we have
  \begin{align*}
    \frac{\chi(\mathcal{T})}{\chi(\mathcal{T}_0)}
    &\leq
      \begin{cases}
        \frac{5+\sqrt{13}}{2\sqrt{3}} \,\,\approx 2.4842  &\qquad \text{for Q-RG},
        \\
        \frac{3+\sqrt{5} }{\sqrt{2}} \,\,\,\,\approx 3.7024 &\qquad \text{for Q-RB}.
      \end{cases}
  \end{align*}
\end{lemma}
\begin{proof}
  We investigate how the shape-regularity constant in~\eqref{def:chi-T} changes if a quadrilateral~$T$ is refined by a red, green or blue refinement, see Figure~\ref{fig:QRG}. 

  We begin with the red refinement. 
  Suppose that~$T$ is split into four quadrilaterals~$T_1,\dots, T_4$. 
  Then we can define $B_{T_j}\colon [0,1]^2 \to T_j$ by composing $B_T$ with the affine maps $S_1, \ldots, S_4$ from $[0,1]^2$ to its four quadrants. Then, with $\norm{\nabla S_i} = \tfrac{1}{2}$ and $\norm{\nabla S_i^{-1}} = 2$   
  the chain rule implies that $\chi(T_j) \leq \chi(T)$.

  Let us continue with the green and blue refinements as in Figure~\ref{fig:QRG}. 
  Those types of refinement are temporary and are replaced by red refinements in further refinements. 
  To estimate the shape regularity of the newly created elements $T_1$, $T_2$ and $T_3$ (or $T_1$ and $T_2$), we compose $B_T$ with an affine linear map from their reference $\widehat{T}_j$ to the (undeformed) elements as in the first line of Figure~\ref{fig:QRG}. 
  As a consequence the shape-regularity constant can increase in the green and blue refinements by a factor given by the shape regularity constant of the elements in Figure~\ref{fig:shape-regularities}.  
  This proves the claim.
\end{proof}

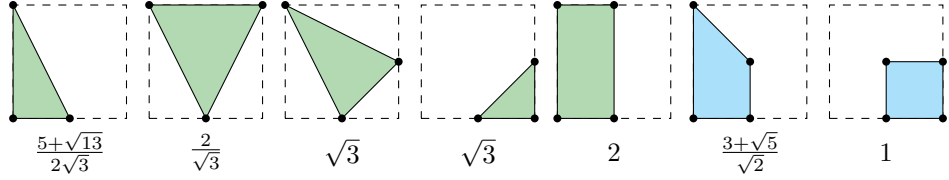
\begin{figure}[t!]
	\begin{tikzpicture}[scale=0.75]
		\begin{scope}[scale=2,shift={(0,0)}]
			\draw[dashed] (0,0) -- (1,0) -- (1,1) -- (0,1) -- cycle;
			\coordinate (A) at (0,0);
			\coordinate (B) at (1/2,0);
			\coordinate (C) at (0,1);
			\draw[fill=rltgreen!30] (A) -- (B) -- (C) -- cycle;
			\fill (A) circle (1pt);
			\fill (B) circle (1pt);
			\fill (C) circle (1pt);
			\node at (1/2,-0.3) {$\frac{5+\sqrt{13}}{2\sqrt{3}}$};
		\end{scope}
		\begin{scope}[scale=2,shift={(1.2,0)}]
			\draw[dashed] (0,0) -- (1,0) -- (1,1) -- (0,1) -- cycle;
			\coordinate (A) at (0,1);
			\coordinate (B) at (1/2,0);
			\coordinate (C) at (1,1);
			\draw[fill=rltgreen!30] (A) -- (B) -- (C) -- cycle;
			\fill (A) circle (1pt);
			\fill (B) circle (1pt);
			\fill (C) circle (1pt);
			\node at (1/2,-0.3) {$\frac{2}{\sqrt{3}}$};
		\end{scope}
		\begin{scope}[scale=2,shift={(2.4,0)}]
			\draw[dashed] (0,0) -- (1,0) -- (1,1) -- (0,1) -- cycle;
			\coordinate (A) at (0,1);
			\coordinate (B) at (1/2,0);
			\coordinate (C) at (1,1/2);
			\draw[fill=rltgreen!30] (A) -- (B) -- (C) -- cycle;
			\fill (A) circle (1pt);
			\fill (B) circle (1pt);
			\fill (C) circle (1pt);
			\node at (1/2,-0.3) {$\sqrt{3}$};
		\end{scope}
		\begin{scope}[scale=2,shift={(3.6,0)}]
			\draw[dashed] (0,0) -- (1,0) -- (1,1) -- (0,1) -- cycle;
			\coordinate (A) at (1/2,0);
			\coordinate (B) at (1,0);
			\coordinate (C) at (1,1/2);
			\draw[fill=rltgreen!30] (A) -- (B) -- (C) -- cycle;
			\fill (A) circle (1pt);
			\fill (B) circle (1pt);
			\fill (C) circle (1pt);
			\node at (1/2,-0.3) {$\sqrt{3}$};
		\end{scope}
		\begin{scope}[scale=2,shift={(4.8,0)}]
			\draw[dashed] (0,0) -- (1,0) -- (1,1) -- (0,1) -- cycle;
			\coordinate (A) at (0,0);
			\coordinate (B) at (1/2,0);
			\coordinate (C) at (1/2,1);
			\coordinate (D) at (0,1);
			\draw[fill=rltgreen!30] (A) -- (B) -- (C) -- (D) -- cycle;
			\fill (A) circle (1pt);
			\fill (B) circle (1pt);
			\fill (C) circle (1pt);
			\fill (D) circle (1pt);
			\node at (1/2,-0.3) {$2$};
		\end{scope}
		\begin{scope}[scale=2,shift={(6.0,0)}]
			\draw[dashed] (0,0) -- (1,0) -- (1,1) -- (0,1) -- cycle;
			\coordinate (A) at (0,0);
			\coordinate (B) at (1/2,0);
			\coordinate (C) at (1/2,1/2);
			\coordinate (D) at (0,1);
			\draw[fill=cyan!30] (A) -- (B) -- (C) -- (D) -- cycle;
			\fill (A) circle (1pt);
			\fill (B) circle (1pt);
			\fill (C) circle (1pt);
			\fill (D) circle (1pt);
			\node at (1/2,-0.3) {$\frac{3+\sqrt{5}}{\sqrt{2}}$};
		\end{scope}
		\begin{scope}[scale=2,shift={(7.2,0)}]
			\draw[dashed] (0,0) -- (1,0) -- (1,1) -- (0,1) -- cycle;
			\coordinate (A) at (1/2,0);
			\coordinate (B) at (1,0);
			\coordinate (C) at (1,1/2);
			\coordinate (D) at (1/2,1/2);
			\draw[fill=cyan!30] (A) -- (B) -- (C) -- (D) -- cycle;
			\fill (A) circle (1pt);
			\fill (B) circle (1pt);
			\fill (C) circle (1pt);
			\fill (D) circle (1pt);
			\node at (1/2,-0.3) {$1$};
		\end{scope}
	\end{tikzpicture}
	
	\caption{Shape-regularity constants of selected elements occurring in the refinement of $[0,1]^2$, see Figure~\ref{fig:QRG}.}
	\label{fig:shape-regularities}
\end{figure}

Ali, Funken, and Schmidt showed the following grading estimate for P-quadrilateral initial meshes in~\cite{AliFunkenSchmidt2022}.
 
\begin{lemma}[{\cite[Lem.~4.5]{AliFunkenSchmidt2022}}]
  \label{lem:AFS-grading}
  Let $\mathcal{T}_0$ be a P-quadrilateral initial mesh.  
  Then, the refinement strategies Q-RG and Q-RB generate a family of meshes with grading~${\gamma \leq 2}$. 
  The constant in~\eqref{eq:hT} for the family of generated meshes depends only on~$\mathcal{T}_0$, specifically on the constant in~\eqref{eq:hT} for~$\mathcal{T}_0$ and on the total number of vertices~$\# \mathcal{V}(\mathcal{T}_0)$.
\end{lemma}

\begin{remark}
  \label{rem:general-grading}
  The proof of Lemma~\ref{lem:AFS-grading} applies directly to quadrilateral initial meshes, since the refinement strategy and the proof of the grading properties only rely on properties of edges and not on the exact shape of the quadrilaterals. 
  Note that in the proof in \cite{AliFunkenSchmidt2022} chains of edges passing through vertices $\mathcal{V}(\mathcal{T}_0)$ are split into subchains containing initial vertices only as their start or end vertex. 
  This strategy is responsible for the fact that the constant in~\eqref{eq:hT} depends on~$\# \mathcal{V}(\mathcal{T}_0)$. 
  We expect that a similar technique as the one developed in~\cite{DieningStornTscherpel2026} improves the constant in~\eqref{eq:hT} in the sense that it only depends on the maximal edge valence of initial vertices~$\mathcal{V}(\mathcal{T}_0)$.
\end{remark}

In the following we shall assume that we have an initial mesh $\mathcal{T}_0$ and an adaptive refinement scheme which generates a family of hybrid meshes $(\mathcal{T}_j)_{j \in \mathbb{N}_0}$ as in Subsection~\ref{sec:mesh} and a grading constant $\gamma \geq 1$ such that
\begin{align}\label{est:grading-family}
  \gamma_{\mathcal{T}_j} \leq \gamma \qquad \text{ for any } j \in \mathbb{N}_0.
\end{align}
By Lemma~\ref{lem:AFS-grading} and Remark~\ref{rem:general-grading} for Q-RB and Q-RG refinements the grading constant is 
\begin{align}\label{eq:grading-Q-ex}
  \gamma_{\textup{Q-RB}} =  \gamma_{\textup{Q-RG}} = 2
\end{align}
for quadrilateral initial meshes.

In~\cite{AliFunkenSchmidt2022} the authors use this grading estimate for Q-RB and Q-RG applied to P-quadrilateral initial meshes to show $W^{1,2}$-stability for $K=2,\dots,9$ in~\cite[Thm.~4.6, 4.9]{AliFunkenSchmidt2022}. 
The special case of P-quadrilateral initial meshes leads to families of meshes consisting of parallelograms, trapezoids and triangles but not general convex quadrilaterals. 
This restriction simplifies the analysis.
We will extend the results in~\cite{AliFunkenSchmidt2022} in several directions: 
We allow general hybrid meshes with grading, which includes in particular Q-RB and Q-RG starting from quadrilateral initial meshes. Moreover, we extend the range of polynomials to $K \geq 2$. 
Finally, we also present $L^p$- and $W^{1,p}$-stability results, see Theorem~\ref{thm:stability} and Corollary~\ref{cor:stability} below.


\section{\texorpdfstring{Stability of the $L^2$-projection}{Stability of the L2-projection}}
\label{sec:stab}

In this section we investigate the stability of the $L^2$-projection onto $\mathcal{L}^1_K(\mathcal{T})$ in Sobolev and Lebesgue norms.

Let $\mathcal{T}$ be a hybrid mesh of a polyhedral domain~$\Omega \subset \RR^2$ as described in Section~\ref{sec:mesh-funct-spac} and let~$K \in \setN$. 
Let~$\Pi \colon L^2(\Omega) \to \mathcal{L}^1_K(\mathcal{T})$ denote the $L^2$-projection onto $\mathcal{L}^1_K(\mathcal{T})$, i.e.,  it satisfies
\begin{align}
  \label{eq:defPi}
  \skp{\Pi u}{v_h} =
  \skp{u}{v_h} \qquad \text{for all $v_h \in \mathcal{L}^1_K(\mathcal{T})$}
\end{align}
for any $u \in L^2(\Omega)$. 
Since $\mathcal{L}^1_K(\mathcal{T}) \subset L^\infty(\Omega)$, the $L^2$-projection is also well-defined on~$L^1(\Omega)$. 
Naturally, $\Pi$ is stable on~$L^2(\Omega)$ meaning that it satisfies $\norm{\Pi u}_2 \leq \norm{u}_2$.
Stability in other norms such as $W^{1,2}(\Omega)$, or $L^p(\Omega)$ and $W^{1,p}(\Omega)$ for $p \neq 2$ requires conditions on the mesh, see~\cite[Sec.~7]{BankYserentant2014}.

Stability in Lebesgue and Sobolev spaces can be obtained from certain weighted $L^2$-estimates of the form
\begin{align}
  \label{eq:weighted}
  \norm{h^{-\alpha} \Pi u}_2 \lesssim
  \norm{h^{-\alpha} u}_2
\end{align}
for some~$\alpha>0$, with $h$ a mesh size function as in \eqref{eq:hT}. 
In particular, the case~$\alpha=1$ implies $W^{1,2}$-stability, see~\cite[Sec.~4.3]{DieningStornTscherpel2021}.

The proof of weighted estimates as in \eqref{eq:weighted} relies on decay properties of~$\Pi$. 
Those can be proven by using a local operator with which one can approximate~$\Pi$.  
This approach has been introduced in~\cite{BankYserentant2014} and was generalised in~\cite{DieningStornTscherpel2021,DieningStornTscherpel2026}. 
In the following we extend this approach to hybrid meshes.

\subsection{Approximating operator}
\label{sec:appr-oper}

In this section we construct a local, self-adjoint operator~$C\colon L^2(\Omega) \to \mathcal{L}^1_K(\mathcal{T})$ that approximates the $L^2$-projection~$\Pi$. 
The definition of~$C$ is analogous to the one in~\cite{DieningStornTscherpel2021} for simplicial meshes.

Our construction is based on the decomposition~\eqref{eq:space-decomposition}. Let $(\phi_i)_{i\in \mathcal{V}}$ denote the nodal basis of~$\mathcal{L}^1_1(\mathcal{T})$, which is a partition of unity. 
For each~$i \in \mathcal{V}$, $\phi_i$ is a weight on~$\Omega_i = \support \phi_i$. 
This allows to define the weighted scalar products 
\begin{align}
  \label{eq:scalarproduct-phii}
  \skp{u}{v}_{\phi_i} \coloneqq \int_{\Omega_i} u v \, \phi_i\dx.
\end{align}
Let $C_i \colon L^{2}(\Omega)\to\mathcal{L}^1_{K-1}(\omega_{i})$ denote the orthogonal projection with respect to $\skp{\cdot}{\cdot}_{\phi_{i}}$ onto $\mathcal{L}^1_{K-1}(\omega_i)$, i.e.,
\begin{align}
  \label{eq:Ci}
  \langle C_{i}u,v_{K-1}\rangle_{\phi_{i}}=\langle u,v_{K-1}\rangle_{\phi_{i}}\quad \text{for all } v_{K-1}\in\ml_{K-1}^{1}(\omega_i),
\end{align}
and for any $u \in L^2(\Omega)$.  
Note that $\Pi$ also satisfies~\eqref{eq:Ci} in the sense that
\begin{align}
  \label{eq:Ci-Pi}
  \langle \Pi u,v_{K-1}\rangle_{\phi_{i}}=\langle u,v_{K-1}\rangle_{\phi_{i}}\quad \text{for all } v_{K-1}\in\ml_{K-1}^{1}(\omega_i),
\end{align}
and for any $u \in L^2(\Omega)$.  
This property of~$\Pi$ is the essential feature for our definition of the operators~$C_i$.

We define the global operator $C\colon L^{2}(\Omega)\rightarrow\LK(\mathcal{T})$ as the (weighted) sum of the local operators $C_{i}$
\begin{align}
  \label{eq:defC}
  C \coloneqq \sum_{i\in\mathcal{V}}\phi_{i}C_{i},
\end{align}
which resembles~\eqref{eq:space-decomposition}. 
Exactly as in~\cite[Thm.~3.2, Lem.~3.4 and Lem.~3.5]{DieningStornTscherpel2021} we obtain the following basic properties of $C$.

\begin{proposition}
  \label{pro:C}
  The operator~$C$ has the following properties: 
  \begin{enumerate}
  \item $C$ is linear and self-adjoint with respect to the $L^2$-inner product~$\skp{\cdot}{\cdot}$.
  \item $C$ is the identity on~$\mathcal{L}^1_{K-1}(\mathcal{T})$.
  \item $u-Cu$ is orthogonal to $\mathcal{L}^1_{K-1}(\mathcal{T})$ for all $u \in L^2(\Omega)$.
  \item If $v \in L^2(\Omega)$ is supported on a collection $L \subset \mathcal{T}$ of elements, then the support of~$Cv$ is at most one layer larger, i.e., it is contained in  $\{ T \in \mathcal{T} \colon \metric(T,L)\leq 1\}$.
  \item For all $u\in L^2(\Omega)$ there holds
    \begin{align*}
      \skp{Cu}{u} = \sum_{i \in \mathcal{V}} \skp{C_iu}{C_iu}_{\phi_i} \leq \sum_{i \in \mathcal{V}} \skp{u}{u}_{\phi_i} =\norm{u}_2^2.
    \end{align*}
  \end{enumerate}
\end{proposition}
To obtain decay estimates for~$\Pi$ as in~\cite{BankYserentant2014,DieningStornTscherpel2021} it remains to estimate from above the condition number\[\operatorname{cond}_2(C|_{\mathcal{L}^1_K(\mathcal{T})}) = \frac{\lambda_{\max}(C|_{\mathcal{L}^1_K(\mathcal{T})})}{\lambda_{\min}(C|_{\mathcal{L}^1_K(\mathcal{T})})}.\] 
By Proposition~\ref{pro:C} we know that $\lambda_{\max}(C|_{\mathcal{L}^1_K(\mathcal{T})})\leq 1$, so we only need an estimate for the smallest eigenvalue of $C|_{\mathcal{L}^1_K(\mathcal{T})}$. 
This is achieved by means of a decomposition operator in the next subsection.

\subsection{Decomposition operator}
\label{sec:decomp-oper}

In this subsection we introduce a decomposition operator of functions in~$\mathcal{L}^1_K(\mathcal{T})$. The general idea goes back to~\cite{BankYserentant2014}. Here, we use a decomposition operator as in~\cite{DieningStornTscherpel2021}, which is based on the function space decomposition~\eqref{eq:space-decomposition}.

In the following we construct decomposition operators $\frD_i \colon \mathcal{L}^1_K(\mathcal{T}) \to \mathcal{L}^1_{K-1}(\omega_i)$ such that
\begin{align}\label{def:decomp}
  v_K = \sum_{i \in \mathcal{V}} \phi_i \frD_i v_K \qquad \text{for all $v_K \in \mathcal{L}^1_K(\mathcal{T})$}.
\end{align}
Exactly as in~\cite[Lem.~3.6 and Prop.~3.14]{DieningStornTscherpel2021} we have the following statement.
\begin{proposition}
  \label{pro:C-lower}
  If there exists $K_1 > 0$ such that
  \begin{align}
    \label{eq:C-lower}
    \sum_{i \in \mathcal{V}} \int_\Omega \phi_i \abs{\frD_i v_K}^2 \dx \leq K_1 \norm{v_K}_2^2
  \end{align}
  for all $v_K \in \mathcal{L}^1_K(\mathcal{T})$,
  then one has that
  \begin{align*}
    \operatorname{cond}_2(C|_{\mathcal{L}^1_K(\mathcal{T})})\leq K_1.
  \end{align*}
\end{proposition}

Let us now construct the decomposition operators~$\frD_i$ for general hybrid meshes.
These operators are defined locally on each element~$T \in \mathcal{T}$. 
Denoting  $\frD_i^T \coloneqq \frD_i|_{T} $ we aim for an estimate of the form
\begin{align}\label{est:DiT-local}
  \sum_{i \in \mathcal{V}(T)}\int_T \phi_i \abs{\frD_i^T v_K}^2 \dx 
  \leq K_1 \int_{T} \abs{v_K}^2 \dx   
  \qquad \text{ for all } v_K \in \mathcal{P}_K(T),
\end{align}
and any $T \in \mathcal{T}$.  
We will show that $\frD_i v_K \in \mathcal{L}^1_{K-1}(\omega_i)$, i.e., that the local definitions are continuous across edges, and that~\eqref{est:DiT-local} implies the global estimate~\eqref{eq:C-lower}.  
In fact, we only need to extend the construction in~\cite{DieningStornTscherpel2021} for triangles to general quadrilaterals.

\subsubsection{Decomposition on triangles}
We begin by reviewing the decomposition for a triangle~$T \in \mathcal{T}$ as introduced in~\cite[Sec.~3.4]{DieningStornTscherpel2021}.
Let $\lambda_0,\lambda_1,\lambda_2$ denote the bary\-centric coordinates of~$T$. 
We consider the polynomials $\lambda^\beta = \lambda_0^{\beta_0} \lambda_1^{\beta_1} \lambda_2^{\beta_2}$ for all multi-indices $\beta=(\beta_0,\beta_1,\beta_2) \in \setN_0^3$. Let $e_0=(1,0,0)^\top$, $e_1=(0,1,0)^\top$ and $e_2=(0,0,1)^\top$ be the canonical unit vectors.  
Let $\{x_\alpha \colon  \alpha \in \setN_0^3,  \abs{\alpha}=K \}$ denote the set of Lagrange nodes of~$\mathcal{L}^1_K(T)$. We define for $i \in \{0,1,2\}$ and $\alpha \in \mathbb{N}^3_0$ with $\abs{\alpha} = K$  the operator
\begin{align}
  \label{eq:def-Di-triangle}
  \frD_i^T(\lambda^\alpha) \coloneqq \frac{\alpha_i}{K} \lambda^{\alpha-e_i}
  =
  \begin{cases}
    \lambda_i(x_\alpha) \lambda^{\alpha-e_i} &\qquad \text{if $\alpha_i >0$},
    \\
    0 &\qquad \text{if $\alpha_i=0,$}
  \end{cases}
\end{align}
with the convention $0 \cdot \lambda^{\alpha-e_i}=0$.
By definition it follows that
\begin{align}
  \label{eq:Di-identity}
  \lambda_i \frD_i^T(\lambda^\alpha) = \lambda_i(x_\alpha) \lambda^\alpha \qquad \text{for $i= \in \{ 0,1,2 \}$\, and $\abs{\alpha}=K$.}
\end{align}
Note that this represents a local version of the decomposition in~\eqref{def:decomp}.
From the proof of Prop.~3.14 in~\cite{DieningStornTscherpel2021} it follows that for any triangle $T \in \mathcal{T}$ one has
\begin{align}
  \label{eq:Di-est-triangle}
  \sum_{i=0}^2\int_T \lambda_i \abs{\frD_i^T v_K}^2 \dx \leq \tfrac{2K+2}{K} \int_{T} \abs{v_K}^2 \dx,
\end{align}
where $i \in \{0,1,2\}$ corresponds to the vertex of $T$ defined by $\lambda_i = 1$. 
Indeed, this verifies~\eqref{est:DiT-local} for triangles. 

\begin{lemma}[\cite{DieningStornTscherpel2021}]\label{lem:K1-triangle}
  If $T \in \mathcal{T}$ is a triangle, then~\eqref{est:DiT-local} holds with minimal constant
  \begin{align*}
    K_1 = \frac{2K+2}{K}.
  \end{align*}
\end{lemma}

\subsubsection{Decomposition on a general convex quadrilateral}
\label{sec:decomp-gener-conv}

We continue with the case that $T \in \mathcal{T}$ is a (strictly) convex quadrilateral.
As in the definition of~$\mathcal{L}_K(T)$ we use the tensor product structure.
For this purpose, we define the local decomposition first for the unit square $\square \coloneqq I \times I$ for  $I\coloneqq [0,1]$. 
Recall from Section~\ref{sec:fct-space} that for $K \in \mathbb{N}$ we have
$ \mathcal{Q}_K(\square) = \mathcal{P}_K(I) \otimes \mathcal{P}_K(I)$. 
We denote by~$\bar{\lambda}_0,\bar{\lambda}_1$ the (one-dimensional) barycentric coordinates of~$I$.
We define the one-dimensional decomposition operator $\frD^I \colon \mathcal{P}_K(I) \to \mathcal{P}_{K-1}(I)$ by
\begin{align}\label{def:DI-1D}
  \frD^I_j(\bar{\lambda}^\alpha) = \frac{\alpha_j}{K} \bar{\lambda}^{\alpha-\bar e_j},
\end{align}
where $\alpha \in \setN_0^2$ with $\abs{\alpha}=K$ and $\bar{e}_0=(1,0)^\top$ and $\bar{e}_1=(0,1)^\top$.
Since this is the special case of~\eqref{eq:def-Di-triangle} for one space dimension, by~\eqref{eq:Di-identity} we have that
\begin{align}
  \label{eq:Di-identity-interval}
  \bar{\lambda}_j \frD_j^I(\bar{\lambda}^\alpha) = \bar{\lambda}_j(x_\alpha) \bar{\lambda}^\alpha \qquad \text{for $j \in \{0,1\}$ and $\abs{\alpha}=K$.}
\end{align}

Let $\mathcal{V}(\square)$ denote the four vertices of~$\square$, namely $(0,0)^\top, (0,1)^\top, (1,0)^\top, (1,1)^\top$.
Then for $i =(i_1,i_2)^\top\in \mathcal{V}(\square)$ we define $\frD^\square_i\colon \mathcal{Q}_K(\square) \to \mathcal{Q}_{K-1}(\square)$ by
\begin{align}
  \label{eq:frD-tensor}
  \frD^\square_i(\bar{\lambda}^\alpha \otimes \bar{\lambda}^\beta)
  \coloneqq
  \frD^I_{i_1}(\bar{\lambda}^\alpha) \otimes \frD^I_{i_2}(\bar{\lambda}^\beta),
\end{align}
where $(f \otimes g)(x_1,x_2) \coloneqq f (x_1)\, g(x_2)$ and $\alpha, \beta \in \mathbb{N}_0^2$ with $\abs{\alpha}=\abs{\beta}=K$.
The decomposition property is inherited from the one-dimensional decomposition.

For general quadrilaterals $ T \in \mathcal{T}$ the local decomposition operators~$\frD_i^T$ with $i \in \mathcal{V}(T)$ are defined as the pullback of the decompositions on~$\square$, i.e.,
\begin{align}\label{def:DT-quad-local}
  \frD^T_i(v_K \circ B_T^{-1})  \coloneqq  (\frD^\square_{B_T^{-1}(i)}v_K) \circ B_T^{-1},
\end{align}
where $B_T\colon [0,1]^2 \to T$ is an affine bilinear transformation as above and $v_K \in \mathcal{Q}_K(\square)$. Recall from Section~\ref{sec:mesh} that $J_T$ is the Jacobian determinant of $B_T$, which is chosen to be positive. 
Hence, by transformation~\eqref{est:DiT-local} is equivalent to
\begin{align}\label{est:K1-gen}
  \sum_{i \in \mathcal{V}(\square)} \int_{\square} \phi_i \abs{\frD_{i}^{\square}(v_K)}^2 {J_T} \dx
  \leq
  K_1  \int_{\square} \abs{v_K}^2 {J_T} \dx,
\end{align}
for any $v_K \in \mathcal{Q}_K(\square)$. 
We now proceed to determine $K_1$ such that~\eqref{est:K1-gen} holds. 

For general convex quadrilaterals $B_T$ is affine bilinear, and hence $J_T$ is affine, i.e., of the form $J_T(x) =  c_0 + c_1 x_1 + c_2 x_2$, for  $x \in \square$ and for $c_i \in \mathbb{R}$ with $i \in \{0,1,2\}$.
Note that for strictly convex quadrilaterals one has~$c_0 > 0$.
To prove~\eqref{est:K1-gen} it suffices to investigate the case $c_0 = 1$, and hence we consider
\begin{align}\label{def:JT}
    J_T(x) =  1 + c_1 x_1 + c_2 x_2.
\end{align}
for $c_1, c_2 \in \mathbb{R}$.
Since $J_T>0$ on $\square$ 
 we only need to show~\eqref{est:K1-gen} for $J_T$ as in~\eqref{def:JT} with
\begin{align}\label{def:C-set}
    (c_1,c_2)^\top \in \mathcal{C} \coloneqq \set{(t_1, t_2)^\top \in \RR^2 \colon t_1 \geq -1, t_2 \geq -1, t_1 + t_2 \ge -1},
\end{align}
as depicted in Figure~\ref{fig:C-set}.
\begin{figure}
\begin{tikzpicture}[scale=1]

  \fill[cyan!20]
  (-1,0) -- (-1,3) -- (3,3) -- (3,-1) -- (0,-1) -- cycle;

  \draw[dashed, thick, blue] (-1,-3) -- (-1,4);      
  \draw[dashed, thick, blue] (-3,-1) -- (4,-1);      
  \draw[dashed, thick, blue] (-3,2) -- (2,-3);       

  \draw[->, thick] (-3,0) -- (3,0) node[right] {$t_1$};
  \draw[->, thick] (0,-3) -- (0,3) node[above] {$t_2$};

  \node[rotate=90, blue] at (-1.3,2) {$t_1=-1$};
  \node[below, blue] at (2.1,-1.1) {$t_2=-1$};
  \node[rotate=-45, blue] at (1.6,-2.2) {$t_1+t_2=-1$};

  \fill[blue] (-1,0) circle (2pt);
  \fill[blue] (0,-1) circle (2pt);

  \node[below left] at (0,0) {$0$};

  \node[below left] at (2,2) {\color{blue}$\mathcal{C}$};

  \draw[dashed, thick, rltgreen] (-2,4) -- (4,-2);
  \node[rotate=-45, rltgreen] at (1.45,1) {$t_1+t_2=\tau-1$};

  \fill[rltgreen] (3,-1) circle (2pt);
  \fill[rltgreen] (-1,3) circle (2pt);

  \fill[pattern=north east lines, pattern color=rltgreen!40]
  (-1,3) -- (3,-1) -- (0,-1) -- (-1,0) -- cycle;

  \node[below left] at (0.7,0.7) {\color{rltgreen}$\mathcal{C}_\tau$};
\end{tikzpicture}
\caption{Convex set $\mathcal{C}$ as defined in~\eqref{def:C-set} and truncated set $\mathcal{C}_{\tau}$ as in~\eqref{def:Ct-set}.}
\label{fig:C-set}
\end{figure}
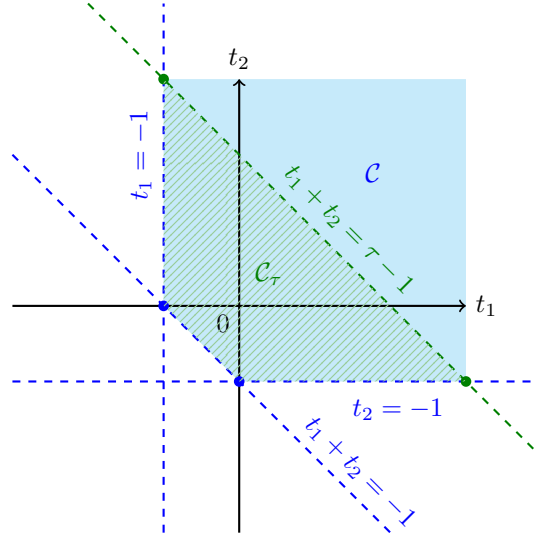
In the special case of a parallelogram $B_T$ is affine, and hence $J_T$ is constant, meaning that $c_1 = c_2 = 0$ in~\eqref{def:JT}. 

In the following we consider the two cases, where $T \in \mathcal{T}$
\begin{enumerate}
  \item  is a parallelogram, i.e.,  $c_1=c_2=0$;
  \item  is a general convex quadrilateral, i.e.,  $(c_1,c_2)^\top \in \mathcal{C}$.
\end{enumerate}

\begin{proposition}
  \label{pro:K1-parallelogram}
  If $T  \in \mathcal{T}$ is a parallelogram, then the smallest constant  $K_1$ in~\eqref{est:DiT-local}
  is given by
  \begin{align*}
    K_1 = \Bigg(\frac{2K+1}{K}\Bigg)^2.
  \end{align*}
  In particular, $K_1$ is the square of the constant of the one-dimensional problem.
\end{proposition}

\begin{proof}
  Recall that~\eqref{est:DiT-local} is equivalent to~\eqref{est:K1-gen}.
  If $T$ is a parallelogram, we have $c_1 = c_2 = 0$ in~\eqref{def:JT}, and thus~\eqref{est:K1-gen} reduces to
  \begin{align*}
    \sum_{i \in \mathcal{V}(\square)} \int_{\square} \phi_i \abs{\frD_{i}^{\square}(v_K)}^2  \dx
     \leq K_1  \int_{\square} \abs{v_K}^2 \dx.
  \end{align*}
  It follows from~\cite[eq.~(3.21)]{DieningStornTscherpel2021} applied in dimension one that
  \begin{align}
    \label{eq:Di-est-intervall}
    \sum_{j=0}^1\int_I \bar{\lambda}_j \abs{\frD_j^I v_I}^2\dx \leq \tfrac{2K+1}{K} \int_{I}\abs{v_I}^2 \dx \quad \text{ for all } v_I \in \mathcal{P}_K(I).
  \end{align}
  It is also shown in~\cite[Rmk.~3.17]{DieningStornTscherpel2021} that this constant is sharp.  
  Let us define the mapping 
  $\frS^I\colon \mathcal{L}_K(I) \to \mathcal{L}_K(I)$ by
  \begin{align}
    \label{eq:frSI}
    \skp{\frS^I v_I}{w_I}_I &\coloneqq
    \sum_{j=0}^1\int_I \bar{\lambda}_j (\frD_j^I v_I)(\frD_j^I w_I)\dx.
  \end{align}
  The operator $\frS^I$ is self-adjoint and positive semidefinite.
  Moreover, it follows from the sharp estimate~\eqref{eq:Di-est-intervall} that
  \begin{align}
    \label{eq:frSI-max}
    \max \sigma(\frS^I) = \tfrac{2K+1}{K},
  \end{align}
  where $\sigma(\frS^I)$ is the spectrum of~$\frS^I$.

 Now, with $\frD_j^\square$ as in~\eqref{eq:frD-tensor} we define the operator
  $\frS^\square\colon \mathcal{L}_K(\square) \to \mathcal{L}_K(\square)$ by
  \begin{align}\label{def:S-square}
    \skp{\frS^\square v_K}{w_K}_\square &\coloneqq
    \sum_{j\in \mathcal{V}(\square)}\int_\square \bar{\lambda}_j (\frD_j^\square v_K)(\frD_j^\square w_K)\dx \quad \text{ for } v_K, w_K \in \mathcal{L}_K(\square).
  \end{align}
  By~\eqref{eq:frD-tensor} and~\eqref{eq:frSI} it follows that $\frS^\square$ has tensor product structure in the sense that
  \begin{align} 
    \label{eq:frS-tensor}
    \bigskp{\frS^\square(\bar{\lambda}^\alpha \otimes \bar{\lambda}^\beta)}{\bar{\lambda}^{\alpha'} \otimes \bar{\lambda}^{\beta'}}_\square = \skp{\frS^I(\bar{\lambda}^\alpha)}{ \bar \lambda^{\alpha'}}_I\, \skp{\frS^I(\bar{\lambda}^\beta)}{\bar \lambda^{\beta'}}_I
  \end{align}
  with $\abs{\alpha}=\abs{\beta}= \abs{\alpha'}=\abs{\beta'}=K$.

  Let $\mathcal{Z} \in \mathcal{P}_K(I)$ denote a complete system of eigenfunctions of~$\frS^I$. 
  Then it follows by~\eqref{eq:frS-tensor} that $\mathcal{Z} \otimes \mathcal{Z}$ is a complete system of eigenfunctions of~$\frS^\square$. 
  Moreover, the spectrum of $\frS^\square$ is $\sigma(\frS^\square) = \frS^I \cdot \frS^I \coloneqq \set{s\,t \colon s,t \in \sigma(\frS^I)}$.
  Since $\frS^\square$ is self-adjoint and positive semidefinite the spectral radius is given by
  \begin{align}
    \label{eq:spectrum-squared}
    \max \sigma(\frS^\square) = \big(\max \sigma(\frS^I)\big)^2 = \Big( \tfrac{2K+1}{K}\Big)^2.
  \end{align}
  Thus, we get the sharp estimate
  \begin{align*}
    \skp{\frS^\square v_K}{v_K}_\square &\leq \Big( \tfrac{2K+1}{K}\Big)^2 \skp{v_K}{v_K}_\square \quad \text{ for all } v_K \in \mathcal{L}_K(\square). 
  \end{align*}
By definition of~$\frS^\square$ this means that
  \begin{align}
    \label{eq:Di-est-parallelogram-pre}
    \sum_{i \in\mathcal{V}(\square)}\int_\square \phi_i \abs{\frD_i^\square (v_K)}^2\dx \leq \Big(\tfrac{2K+1}{K}\Big)^2  \int_{\square}\abs{v_K}^2 \dx
  \end{align}
  for all $v_K \in \mathcal{L}_K(\square)$ with optimal constant. This proves the claim.
\end{proof}

Let us consider another special case, which will be used in the estimate for general convex quadrilaterals below.

\begin{lemma}[auxiliary estimate]\label{lem:aux}
  If $J_T(x) = x_1$, then~\eqref{est:K1-gen} holds with 
  \begin{align*}
K_1 = \Bigg(\frac{2K+1}{K}  \Bigg) \Bigg( \frac{2K + 2}{K}\Bigg),
  \end{align*}
  and the estimate is sharp.
\end{lemma}
\begin{proof}
  For $J_T(x) = x_1$ the estimate~\eqref{est:K1-gen} reduces to
  \begin{align}\label{est:aux-1}
    \sum_{i \in \mathcal{V}(\square)} \int_{\square} \phi_i \abs{\frD_{i}^{\square}(v_K)}^2 x_1 \dx
    \leq
    {K}_1  \int_{\square} \abs{v_K}^2 x_1 \dx,
  \end{align}
  for any $v_K \in \mathcal{L}_K(T)$.

  Because the polynomial space $\mathcal{Q}_K(\square) = \mathcal{P}_K(I) \otimes \mathcal{P}_K(I)$, the decomposition operator $\frD^\square$, and the linear weight $x_1$ possess tensor product structure, the complete system of generalised eigenfunctions is spanned by the tensor products of the eigenfunctions in one dimension. 
  This guarantees that the eigenfunction in dimension~$d = 2$  with the largest eigenvalue is exactly the product of the eigenfunctions of the two independent one-dimensional problems.  
  This allows us to reduce the statement to the one for separable functions~$v_K(x) = f(x_1) g(x_2)$ with~$f,g \in \mathcal{P}_K(I)$. 
  
  For such $v_K$ the left-hand side of~\eqref{est:aux-1} can be represented by
  \begin{align*}
    &	\sum_{i \in \mathcal{V}(\square)} \int_{\square} \phi_i \abs{\frD_{i}^{\square}(v_K)}^2 x_1 \dx \\
    & \qquad  \qquad = \sum_{i \in \mathcal{V}(\square)}
    \int_{I} \bar{\lambda}_{i_1} \abs{\frD^I_{i_1}(f)}^2 x_1 \dx_1 \int_{I}  \bar{\lambda}_{i_2}\abs{\frD^I_{i_2}(g)}^2 \dx_2  \\
    & \qquad \qquad =  \left( \sum_{j = 0}^1 \int_{I} \bar{\lambda}_j \abs{\frD^I_{j}(f)}^2 x_1 \dx_1 \right)
    \left( \sum_{\ell = 0}^1 \int_{I} \bar{\lambda}_{\ell} \abs{\frD^I_{\ell}(g)}^2  \dx_2 \right). 
  \end{align*}
  Similarly, the integral on the right-hand side of~\eqref{est:aux-1} can be written as
  \begin{align*}
    \int_{\square} \abs{v_K}^2 x_1 \dx =
    \int_{I} \abs{f}^2 x_1 \dx_1    \int_{I} \abs{g}^2 \dx_2.
  \end{align*}
  For the integral in~$x_2$ by~\eqref{eq:Di-est-intervall} we have that
  \begin{align}\label{est:x2}
    \sum_{\ell = 0}^1 \int_{I} \bar{\lambda}_{\ell} \abs{\frD^I_{\ell}(g)}^2  \dx_2 \leq \frac{2K+1}{K}       \int_{I} \abs{g}^2 \dx_2.
  \end{align}
  Thus, it remains to find $\tilde{K}_1>0$ such that
  \begin{align}\label{est:x1}
    \sum_{j = 0}^1 \int_{I} \bar{\lambda}_j \abs{\frD^I_{j}(f)}^2 x_1 \dx_1 \leq  \tilde{K}_1 \int_{I} \abs{f}^2 x_1 \dx_1,
  \end{align}
  for any $f \in \mathcal{P}_K(I)$.  
  Noting that $x_1 = \int_{0}^{x_1} 1 \dx_2$ we aim to rewrite~\eqref{est:x1} as a corresponding decomposition estimate on the auxiliary triangle $\tilde T$ with vertices $a_0 = (0,0)^\top$, $a_1 = (1,0)^\top$ and $a_2 = (1,1)^\top$.
  Let us define the trivial extension of $f$ to $\tilde T$ by~$\tilde f (x) \coloneqq f(x_1)$. Then one has
  \begin{align*}
    \int_{I} \abs{f}^2 x_1 \dx_1 =
    \int_{\tilde T} \abs{\tilde f}^2 \dx.
  \end{align*}
  In the following, the barycentric coordinates~$\bar{\lambda}_0, \bar{\lambda}_1$ are understood as extensions to~$\tilde T$.
  Let~$\lambda_0, \lambda_1, \lambda_2$ be the corresponding barycentric coordinates of $\tilde T$. 
  Then~$\bar{\lambda}_0 = \lambda_0$ and $\bar{\lambda}_1 = \lambda_1 + \lambda_2$. 
  In combination with~\eqref{def:DI-1D} after a short calculation one arrives at 
  \begin{align}
    \label{eq:D-with-tilde}
    \begin{aligned}
      (\frD_0^I f)(x_1) &=  (\frD_0^{\tilde T} \tilde f)(x),
      \\
      (\frD_1^I f)(x_1) &=  (\frD_1^{\tilde T} \tilde f)(x) =
      (\frD_2^{\tilde T} \tilde f)(x).
    \end{aligned}
  \end{align}
  Overall, \eqref{est:x1} is equivalent to
  \begin{align}
    \label{est:x1-rewritten}
    \sum_{j=0}^2 \int_{\tilde T} \lambda_j \abs{\frD_j^{\tilde T} \tilde{f}}^2 \dx \leq \tilde{K}_1 \int_{\tilde T} \abs{\tilde{f}}^2 \dx.
  \end{align}
  This is exactly the local triangular estimate in Lemma~\ref{lem:K1-triangle} with~$v_K$ replaced by~$\tilde{f}$.  
  Hence, since $\tilde f \in \mathcal{P}_K(\tilde T)$, Lemma~\ref{lem:K1-triangle} provides us with the upper bound~$\tilde{K}_1 \leq \frac{2K+2}{K}$. 
  
  It remains to argue that the bound is sharp, even though $\tilde f$ in \eqref{est:x1-rewritten} only depends on $x_1$.
  For this we use the technique in the proof of Lemma~\ref{lem:K1-triangle} established  in~\cite[Prop.~3.16]{DieningStornTscherpel2021}.  
  Let $\mathcal{Z}_K$ denote the subspace of polynomials of~$\mathcal{P}_K(\tilde T)$ that are $L^2$-orthogonal to $\mathcal{P}_{K-1}(\tilde T)$. In~\cite[Prop.~3.16]{DieningStornTscherpel2021} it is proved that each~$z_K \in \mathcal{Z}_K \setminus \{0\}$ is an eigenfunction that maximises the Rayleigh quotient induced by~\eqref{est:x1-rewritten}. 
  Thus, it suffices to find one function $z_K \in \mathcal{Z}_K$ that only depends on~$x_1$. 
  The shifted Jacobi polynomial~$\frac{1}{x_1} \frac{d^K}{d x_1^K} (x_1^{K+1}(1-x_1)^K)$ is such a function.  
  This proves that the upper bound in~\eqref{est:x1-rewritten} is sharp and hence that $K_1 = \tilde K_1 \frac{2K+1}{K}$ in~\eqref{est:K1-gen} and~\eqref{est:DiT-local} is sharp.
  \renewcommand{\qedsymbol}{\emph{(Quad erat demonstrandum)}}
\end{proof}

With this auxiliary result we are in the position to find a bound $K_1$ for general convex quadrilaterals.

\begin{proposition}\label{prop:gen-quad}
  If $T \in \mathcal{T}$ is a strictly convex quadrilateral, then the estimate~\eqref{est:DiT-local}
  holds with
  \begin{align*}
    K_1 = \Bigg(\frac{2K+1}{K} \Bigg) \Bigg(  \frac{2K+2}{K} \Bigg).
  \end{align*}
  In particular, $K_1$ is the product of the constant of the one-dimensional problem and the one of a triangle. 
  The constant $K_1$ is the smallest constant such that~\eqref{est:DiT-local}
  holds for all convex quadrilaterals.
\end{proposition}
\begin{proof}
  Let $T \in \mathcal{T}$ be a strictly convex quadrilateral. 
  Recall that~\eqref{est:DiT-local} is equivalent to~\eqref{est:K1-gen} with $J_T = 1+c_1x_1 +c_2 x_2$ and $c \in \mathcal{C}$, see~\eqref{def:JT} and~\eqref{def:C-set}.
  To find the optimal constant $K_1$ in~\eqref{est:K1-gen} we need to maximise a generalised Rayleigh quotient.

  \textit{Step 1 (derivation of generalised Rayleigh quotient):}
  We denote $N \coloneqq \dim \mathcal{Q}_{K}(\square)$, and we consider the weight functions
  \begin{align*}
    \omega_0(x) = 1 \quad \text{ and } \quad \omega_r(x) = x_r \quad \text{ for } r \in \{1,2\},
  \end{align*}
  on $\square$. 
  Furthermore, for a basis $(\psi_{j})_{j \in \{1, \ldots, N\}}$ of $\mathcal{Q}_K(\square)$ we define the symmetric matrices $A_r, M_r \in \mathbb{R}^{N \times N}$ for $r \in \{0,1,2\}$ representing the weighted integrals in~\eqref{est:K1-gen} with entries, for $j,\ell \in \{1, \ldots, N\}$,
  \begin{align*}
    (A_r)_{j,\ell} & = \sum_{i \in \mathcal{V}(\square)} \int_{\square} \phi_i \frD_{i}^{\square}(\psi_j)\, \frD_{i}^{\square}(\psi_\ell)\,  \omega_r \dx, \\
        (M_r)_{j,\ell}& =
  \int_{\square} \psi_j \,\psi_{\ell} \, \omega_r  \dx.
  \end{align*}
  Then the Rayleigh quotient corresponding to  estimate~\eqref{est:K1-gen} is
  \begin{align}\label{def:R}
    R(z;c)  \coloneqq \frac{z^\top A_0 z + c_1 z^\top A_1 z + c_2 z^\top A_2 z}{z^\top M_0 z + c_1 z^\top M_1 z + c_2 z^\top M_2 z},
  \end{align}
  for $z\in \setR^N \setminus \{0\}$ and for $c  = (c_1,c_2)^\top \in \mathcal{C}$. In particular, \eqref{est:K1-gen} is equivalent to
  \begin{align}\label{est:equiv-K1}
    \sup_{z \in \setR^N \setminus \{0\}}\, \sup_{c \in \mathcal{C}}  \,  R(z;c)
    =  \sup_{c \in \mathcal{C}}  \, \sup_{z \in \setR^N \setminus \{0\}}\, R(z;c) \leq K_1.
  \end{align}
  Since the mass matrix $M_0$ is positive definite, and $M_1, M_2$ are positive semidefinite, the denominator of $R$ is positive for any $z \in \mathbb{R}^N \setminus \{0\}$, and thus $R$ is continuous.

  \textit{Step 2 (certain values of $R$): }
  For later use let us record that
  \begin{align}
    \label{est:R-2termsa}
    R\big(z;(-1,0)^\top\big) &= \frac{z^\top A_0 z - z^\top A_1 z}{z^\top M_0 z - z^\top M_1 z}  &\text{ corresponding to } J_T(x)& = 1 - x_1.
  \end{align}
  Let us define $R(z;(\infty,-1)^\top) \coloneqq \lim_{c_1 \to \infty} R(z;(c_1,-1)^\top)$. Then
  \begin{align}
    \label{est:R-2termsb}
    R\big(z;(\infty,-1)^\top\big)
    &= \frac{ z^\top A_1 z}{ z^\top M_1 z} &\text{ corresponding to } J_T(x) &= x_1.
  \end{align}
  By invariance of~\eqref{est:K1-gen} under the transformation $x_1 \mapsto 1 - x_1$ it follows that
  \begin{align}\label{eq:R-val}
    R\big(z;(\infty,-1)^\top\big) =  R\big(z;(-1,0)^\top\big)  =  \sup_{z \in \setR^N \setminus \{0\}}\, \frac{ z^\top A_1 z}{ z^\top M_1 z}  \eqqcolon \hat K_1.
  \end{align}
  By definition of $A_1, M_1$ the constant $\hat K_1>0$ is the smallest constant such that
  \begin{align}\label{est:K1tilde}
    \sum_{i \in \mathcal{V}(\square)} \int_{\square} \phi_i \abs{\frD_{i}^{\square}(v_K)}^2 x_1 \dx
    \leq
    \hat K_1  \int_{\square} \abs{v_K}^2 x_1 \dx,
  \end{align}
  for any $v_K \in \mathcal{Q}_K(\square)$ and
  Lemma~\ref{lem:aux} yields that $\hat K_1  =\big(\frac{2K+1}{K} \big) \big( \frac{2K+2}{K} \big)$.

  \textit{Step 3 (reduction by quasiconvexity): }
  For fixed~$z \in \setR^N \setminus\{0\}$ consider the mapping~$R(z;\cdot) \colon \mathcal{C} \to \mathbb{R}$. 
  We want to show that it is quasiconvex in the sense that its sublevel sets~$\{c \in \mathcal{C} \colon R(z;c) \leq \mu \}$ are convex for any~$\mu \in \mathbb{R}$, see~\cite[Thm.~2.2.3]{CambiniMartein2009}.
  Since the denominator of $R$ is positive, the sublevel sets are defined by
  \begin{align*}
    z^\top A_0 z + c_1 z^\top A_1 z + c_2 z^\top A_2 z \leq \mu( z^\top M_0 z + c_1 z^\top M_1 z + c_2 z^\top M_2 z).
  \end{align*}
  Since this is an affine condition on $c = (c_1,c_2)^\top \in \mathcal{C}$, the sublevel set is in particular convex, and thus $R(z;\cdot)$ is quasiconvex.

  In order to bound $R(z;\cdot)$ above on~$\mathcal{C}$, cf.~\eqref{est:equiv-K1}, we apply the fact that any quasiconvex function over a compact convex domain assumes its maximum in one of the extremal points of the domain, see~\cite[Thm.~4.6.3]{CambiniMartein2009}.
  Since $\mathcal{C}$ defined in~\eqref{def:C-set} is not bounded, we truncate it and define the set
  \begin{align}\label{def:Ct-set}
    \mathcal{C}_\tau \coloneqq \mathcal{C} \cap \{(t_1, t_2)^\top \colon t_1 + t_2 \leq \tau - 1\} \quad \text{ for } \tau \geq 1,
  \end{align}
  which is compact and convex, see Figure~\ref{fig:C-set}.
  The extremal points of $\mathcal{C}_\tau$ are
  \begin{align*}
    W_1 = (-1, 0), \;\; W_2 = (0, -1), \;\;  W_3 = (\tau, -1), \quad \text{ and } \quad  W_4 = (-1, \tau).
  \end{align*}
  Hence, after taking the supremum over~$z$ and using the symmetries of the square, we obtain
  \begin{align*}
    \sup_{z \in \setR^N \setminus \{0\}}\sup_{c \in \mathcal{C}_\tau} R(z;c)
    &\leq \max_{j=1,\dots,4} \sup_{z \in \setR^N \setminus \{0\}} R(z;W_j) \\
    &= \max \bigset{ \hat K_1,\, \sup_{z \in \setR^N \setminus \{0\}} R\big(z;(\tau,-1)^\top\big)}.
  \end{align*}
  By continuity of $R(z;\cdot)$ and~\eqref{eq:R-val} we have
  \begin{align*}
    \lim_{\tau \to \infty} \sup_{z \in \setR^N \setminus \{0\}} R\big(z;(\tau,-1)^\top\big)
    = \sup_{z \in \setR^N \setminus \{0\}} R\big(z;(\infty,-1)^\top\big)
    = \hat K_1.
  \end{align*}
  Since the sets $\mathcal{C}_\tau$ are increasing with $\lim_{\tau \to \infty}\mathcal{C}_\tau = \mathcal{C}$ we have
  \begin{align*}
    \sup_{z \in \setR^N \setminus \{0\}}\sup_{c \in \mathcal{C}} R(z;c)
    = \lim_{\tau \to \infty}
    \sup_{z \in \setR^N \setminus \{0\}}\sup_{c \in \mathcal{C}_\tau} R(z;c)
    \leq \hat K_1.
  \end{align*}
  Since $(-1,0)^\top \in \mathcal{C}$, this estimate is sharp. 
  Taking the supremum over~$z \in \RRN \setminus \set{0}$ and noting that $\hat K_1  = \big(\frac{2K+1}{K} \big)\big( \frac{2K+2}{K} \big)$ proves the claim.
\end{proof}

\begin{remark}\label{rmk:convex}
  The estimates in Proposition~\ref{prop:gen-quad} still hold if the quadrilateral $T$ is degenerate and $T$ is a triangle.
  In the Jacobian $J_T$ of the corresponding bilinear mapping $B_T$ one can still ensure that $c_0 >0$ by rotating the vertex labels if needed.
\end{remark}


\subsection{Global decomposition}
\label{sec:global-decomposition}

 Let us now define the global decomposition operators $\frD_i\colon  \mathcal{L}^1_K(\mathcal{T}) \to \mathcal{L}^1_{K-1}(\omega_i)$ for vertex $i \in \mathcal{V}$ satisfying \eqref{def:decomp}. 
 For this purpose, we patch together the local decomposition operators, introduced in the previous subsection. 
For $T \in \mathcal{T}$ with $i \in T$ let $j \in \mathcal{V}(T)$ be the local index of the vertex~$i$ in~$T$. 
We define $\frD_i(v_K)$ for $v_K \in \mathcal{L}^1(\mathcal{T})$ on $T$ by 
\begin{align}\label{def:globalD}
 (\frD_iv_K)|_T \coloneqq
 \frD_j^T (v_K|_T),
\end{align}
with $\frD_j^T$ as defined in~\eqref{eq:def-Di-triangle} and~\eqref{def:DT-quad-local} for triangles and for general (convex) quadrilaterals, respectively.
This local definition ensures that~$\frD_i\colon  \mathcal{L}^1_K(\mathcal{T}) \to \mathcal{L}^0_{K-1}(\omega_i)$. 
The following proposition shows that $\frD_i$  maps to~$\mathcal{L}^1_{K-1}(\omega_i)$, i.e., to continuous functions. 

\begin{proposition}
  \label{pro:Di-global}
  The decomposition operators~$\frD_i$ defined in \eqref{def:globalD} map $\mathcal{L}^1_K(\mathcal{T}) \to \mathcal{L}^1_{K-1}(\omega_i)$ for $i \in \mathcal{V}$. 
  Moreover, $\phi_i \frD_i$ maps $\mathcal{L}^1_K(\mathcal{T}) \to \mathcal{L}^1_K(\mathcal{T})$ and satisfies 
  \begin{align*}
    v_K &=\sum_{i \in \mathcal{V}} \phi_i \frD_i(v_K)
  \end{align*}
  for all $v_K \in \mathcal{L}^1_K(\mathcal{T})$. 
  Moreover, there exists a constant $K_1>0$ such that
  \begin{align}\label{est:K1-bound}
    \sum_{i \in \mathcal{V}}\int_\Omega \phi_i \abs{\frD_i v_K}^2\dx \leq K_1
    \norm{v_K}_2^2 \qquad \text{ for any }v_K \in \mathcal{L}^1_{K}(\mathcal{T}). 
  \end{align}
  The smallest constant~$K_1$ is given by
  \begin{align}\label{def:K1-const}
    K_1 =
    \begin{cases}
      \frac{2K+2}{K} &\qquad \text{if $\mathcal{T}$ is simplicial},
      \\
      \big(\frac{2K+1}{K}\big)^2 &\qquad \text{if $\mathcal{T}$ is SP-hybrid},
      \\
      \big( \frac{2K+1}{K} \big) \big( \frac{2K+2}{K}\big) &\qquad \text{if $\mathcal{T}$ is general hybrid}.
    \end{cases}
  \end{align}
  In particular, we have that~$\operatorname{cond}_2(C|_{\mathcal{L}^1_K(\mathcal{T})}) \leq K_1$ for $C$ as defined in~\eqref{eq:defC}.  
\end{proposition}
\begin{proof}
  It has been shown in~\cite[Lem.~3.10]{DieningStornTscherpel2021} that the restriction of the decomposition operators to edges only depends on the value of the argument on the edges.  Due to the tensor product structure this feature is also valid for edges of quadrilaterals.  Thus, as in~\cite[Prop.~3.1]{DieningStornTscherpel2021} it follows that the functions $\frD_i^Tv_K$ are continuous on~$\Omega_i$.  For the local estimates of the form~\eqref{est:DiT-local} for triangles we use Lemma~\ref{lem:K1-triangle}, for triangles and parallelograms we additionally use Proposition~\ref{pro:K1-parallelogram} and in the case of general hybrid meshes we also use Proposition~\ref{prop:gen-quad}.  The global estimate~\eqref{est:K1-bound} follows by summation over~$T \in \mathcal{T}$.  Finally, we estimate $\operatorname{cond}_2(C|_{\mathcal{L}^1_K(\mathcal{T})})$ using Proposition~\ref{pro:C-lower}.
\end{proof}

\begin{remark}\label{rmk:general-dim}
  In general dimensions for meshes consisting of $d$-simplices and parallelotopes an analogous construction yields 
  \begin{align*}
    \operatorname{cond}_2(C|_{\mathcal{L}^1_K(\mathcal{T})})\leq \big(\tfrac{2K+1}{K}\big)^d.
  \end{align*}
  Note that for $d \geq 2$ the dependence of the bound on the dimension $d$ is worse than for simplicial meshes, see~\cite{DieningStornTscherpel2021}.  Below, for the $W^{1,2}$-stability of the $L^2$-projection we rely on the condition $\operatorname{cond}_2(C|_{\mathcal{L}^1_K(\mathcal{T})})<9$.  However, the condition $(\frac{2K+1}{K})^d < 9$ never holds for $d\geq 4$ and for $d=3$ it is satisfied only for $K \geq 13$.  Thus, the estimate is not very useful for $d > 2$.
\end{remark}

\begin{remark}
  The estimate in Proposition~\ref{pro:Di-global} for $K=1$ and for parallelograms is sharp. 
  Let $\mathcal{T}$ consist of the square~$T = [0,1]^2$.
  Let $v_K \in \mathcal{L}^1_1(T)$ with $v_K((0,0)^\top)=1$, $v_K((0,1)^\top)=-1$, $v_K((1,0)^\top)=-1$, $v_K((1,1)^\top)=1$. 
  Then $\skp{Cv_K}{v_K} = \frac 19 \norm{v_K}_2^2$ and therefore $\operatorname{cond}_2(C|_{\mathcal{L}^1_1(\mathcal{T})}) = 9$. 
\end{remark}

\subsection{Decay estimates}
\label{sec:decay-estimates}

Based on Proposition~\ref{pro:C} the technique of~\cite{BankYserentant2014} allows to derive a decay estimate for the $L^2$-projection~$\Pi \colon L^2(\Omega) \to \mathcal{L}^1_K(\mathcal{T})$ with explicit decay rate. 
The rate is independent of the shape regularity and only depends on the polynomial degree $K$. 
Decay estimates, in turn, are key to establish $L^p$- and $W^{1,p}$-stability estimates with the maximal function approach developed in~\cite{DieningStornTscherpel2021}, see Section~\ref{sec:stab-res} below.  

\begin{proposition}[decay estimate]
  \label{pro:decay}
  Let $L,L' \subset \mathcal{T}$ be a collection of elements and let $u \in L^2(\Omega)$. Then with $K_1$ from Proposition~\ref{pro:Di-global} and decay parameter
  \begin{align}\label{def:q}
    q \coloneqq  \frac{\sqrt{K_1}\,-1}{\sqrt{K_1}\,+1}
  \end{align}
  then
  \begin{align}
    \label{eq:decay2}
    \bignorm{\indicator_L \Pi (\indicator_{L'} u)}_2
    \leq
    \min \bigset{ 2
    q^{\metric(L,L')-1},1}
    \norm{\indicator_{L'} u}_2.
  \end{align}
\end{proposition}
\begin{proof}
  The statement can be found in~\cite[Lem.~3.1]{BankYserentant2014} and is a consequence of Proposition~\ref{pro:Di-global}. 
  For an abstract framework see~\cite[Sec.~2.2]{DieningStornTscherpel2021}, for which we have verified conditions (C1)--(C4) in Proposition~\ref{pro:C} and \ref{pro:Di-global} above. 
  Due to~\cite[Prop.~2.2]{DieningStornTscherpel2021} we obtain a decay in the form of~\eqref{eq:decay2} with $q$ replaced by
  \begin{align*}
    \bar{q} \coloneqq  \frac{\sqrt{\kappa}\,-1}{\sqrt{\kappa}\,+1} \quad \text{where} \quad \kappa \coloneqq \operatorname{cond}(C|_{\mathcal{L}^1_K(\mathcal{T})}).
  \end{align*}
  The estimate $\kappa \leq K_1$ in Proposition~\ref{pro:Di-global} proves the claim.
\end{proof}

\begin{remark}
  \label{rem:final-q}
  Let us summarise the decay rates~$q$ established in Proposition~\ref{pro:decay} for the refinement routines Q-RG and Q-RB for general quadrilateral initial meshes $\mathcal{T}_0$. 
  In this case, by Lemma~\ref{lem:refinement-properties} and Proposition~\ref{pro:decay}  we have that~\eqref{eq:decay2} holds with
  \begin{align}
    \label{eq:final-q}
    q =
    \frac{\sqrt{(2K+1)(2K+2)} - K}{\sqrt{(2K+1)(2K+2)} + K}. 
  \end{align}
  In the special case of Q-RG applied to a P-quadrilateral initial mesh (resulting in SP-hybrid meshes, see Table~\ref{tab:refinement-rules}), with $K_1 = \big( \frac{2K+1}{K} \big)^2$  we obtain the smaller value
  \begin{align*}
    q =
    \frac{K+1}{3K+1}. 
  \end{align*}
  For both cases, and also for simplicial meshes, the decay rate $q$ as function of $K$ is displayed in Figure~\ref{fig:decay}, see Table~\ref{tab:decay} for the values. 
  Note that $q(K)$ is monotonic decreasing in $K$. 
  Hence $q$ is bounded above by $q(2)< \tfrac{1}{2}$ for any $K \geq 2$.
\end{remark}

\begin{figure}
  \begin{tikzpicture}[scale = 0.8]
    \begin{axis}[
      width=10cm,
      height=7cm,
      xlabel={$K$},
      ylabel={$q$},
      xmin=1, xmax=10,
      ymin=0.15, ymax=0.6,
      xtick={1,2,3,4,5,6,7,8,9,10},
      legend pos=north east,
      xmajorgrids=true,
      ymajorgrids=true,
      grid style=dotted, 
      legend cell align=left,
      legend style={align=left}
      ]

      \addplot[
      color=blue,
      mark=*,
      domain=1:10,
      thick,
      samples=10 
      ]
      {(sqrt(2*x + 2) - sqrt(x)) / (sqrt(2*x + 2) + sqrt(x))};
      \addlegendentry{\scriptsize simplicial mesh}

      \addplot[
      color=green!70!black, 
      thick,
      mark=square*,
      domain=1:10,
      samples=10
      ]
      {(x + 1) / (3*x + 1)};
      \addlegendentry{\scriptsize SP-hybrid mesh
      	}
      	
      \addplot[
      color=red,
      mark=triangle*,
      thick,
      domain=1:10,
      samples=10
      ]
      {(sqrt((2*x + 1)*(2*x + 2)) - x) / (sqrt((2*x + 1)*(2*x + 2)) + x)};
      \addlegendentry{\scriptsize general hybrid mesh}
      
      \addplot[
      color=black,
      dotted, 
      very thick,      domain=1:10,
      samples=10
      ]
      {0.5};
      \addlegendentry{\scriptsize $1/2$}
    \end{axis}
  \end{tikzpicture}
  \caption{Decay rate $q$ plotted over polynomial degree $K$ for different types of meshes, see Remark~\ref{rem:final-q}. 
  	Note that Q-RG generates SP-hybrid and Q-RB generates general hybrid meshes for P-quadrilateral initial mesh, and that both generate general hybrid meshes for quadrilateral initial mesh, see Table~\ref{tab:refinement-rules}.}
  \label{fig:decay}
\end{figure}
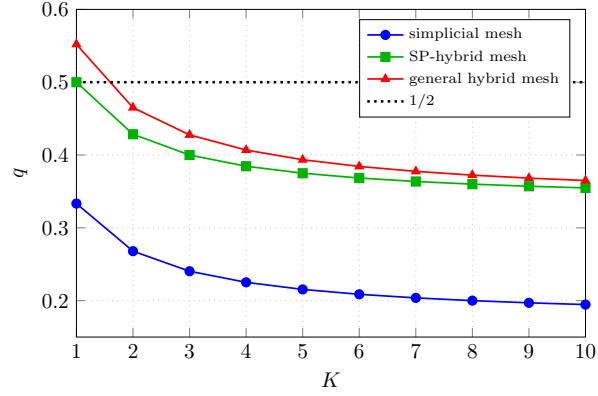

\begin{table}[ht!]\small
  \centering
  \begin{TAB}(r)[1pt]{|c|c|c|c|}{|c|cccccccccc|}
    $K$ & \;\; simplicial $\mathcal{T}$  \;\; &  \;\;  SP-hybrid $\mathcal{T}$ \;\; & \;\;  general hybrid $\mathcal{T}$ \;\; \\
    $1$  & $0.333333$ & $(0.500000)$ & $(0.551982)$ \\
    $2$  & $0.267949$ & $0.428571$ & $0.465042$ \\
    $3$  & $0.240408$ & $0.400000$ & $0.427662$ \\
    $4$  & $0.225148$ & $0.384615$ & $0.406829$ \\
    $5$  & $0.215438$ & $0.375000$ & $0.393540$ \\
    $6$  & $0.208712$ & $0.368421$ & $0.384323$ \\
    $7$  & $0.203777$ & $0.363636$ & $0.377555$ \\
    $8$  & $0.200000$ & $0.360000$ & $0.372373$ \\
    $\vdots$ & $\vdots$ & $\vdots$ & $\vdots$ \\
    $\infty$ & $\frac{\sqrt{2} -1}{\sqrt{2}+1}$ & $\tfrac{1}{3}$ & $\tfrac{1}{3}$ \\
  \end{TAB}
  \caption{Decay rate $q$ for different types of (hybrid) meshes $\mathcal{T}$ and polynomial degree $K$. If $\gamma = 2$, then $W^{1,2}$-stability holds, provided that $q< 1/2$, see Corollary~\ref{cor:stability}~\ref{itm:stability-W12}. Values greater than or equal~$\frac 12$ are in brackets.}
  \label{tab:decay}
\end{table}

\subsection{Stability results}
\label{sec:stab-res}
Based on the decay estimate in Proposition~\ref{pro:decay} we obtain the following stability estimates directly as in \cite[Sec.~4]{DieningStornTscherpel2021}. 

\begin{theorem}
  \label{thm:stability}
  Let $\mathcal{T}$ be a hybrid mesh with grading~$\gamma_{\mathcal{T}}$ and $\rho \in \mathcal{L}^0_0(\mathcal{T})$ be a weight with grading~$\gamma_\rho$. Let $\Pi\colon L^2(\Omega) \to \mathcal{L}^1_K(\mathcal{T})$ denote the $L^2$-projection with $K \in \mathbb{N}$. 
  Let $1 \leq p \leq \infty$  and let $q$ be as in~\eqref{def:q}, see also Table~\ref{tab:decay}.
  \begin{enumerate}
  \item  If
    $\gamma_\rho \gamma_{\mathcal{T}}^{2\abs{\frac 12 - \frac 1p}} < q^{-1}$,
    then there exists a constant~$c_0> 0$ such that
    \begin{align*}
      \bignorm{\rho \Pi u}_p &\leq c_0 \bignorm{\rho u}_p \qquad \text{for all $u \in L^p(\Omega)$.}
    \end{align*}
  \item  If $
    \gamma_\rho \gamma_{\mathcal{T}}^{1+2\abs{\frac 12 - \frac 1p}}  < q^{-1}$, then there exists a constant $c_1 > 0$ such that
    \begin{align*}
      \bignorm{\rho \nabla \Pi u}_p &\leq c_1 \bignorm{\rho \nabla u}_p \qquad \text{for all $u \in W^{1,p}(\Omega)$.}
    \end{align*}
  \end{enumerate}
  The constants~$c_0, c_1$ only depend on~$K,\gamma_\rho, \gamma_{\mathcal{T}},\chi(\mathcal{T})$ and the constants in~\eqref{eq:hT}.
\end{theorem}
\begin{proof}
  The proof follows exactly as the one of Theorems~4.12 and 4.13 of~\cite{DieningStornTscherpel2021}. 
  Note that the volume decay estimate in Lemma~4.10 therein generalises to shape-regular hybrid meshes, since~\eqref{eq:shape-regular} holds. 
  Moreover, standard inverse estimates and the local stability and approximability of a quasi-interpolation operator is employed, which is available in~\cite{BernardiGirault1998} and~\cite[Sec.~2.4]{ApelMelenk2017}.
\end{proof}

Recall that for any family of hybrid meshes generated by the adaptive refinement routines Q-RB and Q-RG starting from a quadrilateral mesh $\mathcal{T}_0$ the shape-regularity constant is uniformly bounded by~$4\chi(\mathcal{T}_0)$, see Lemma~\ref{lem:shape-regular-Q-RG-Q-RB}. 
Moreover, by Lemma~\ref{lem:AFS-grading} and Remark~\ref{rem:general-grading} the grading constant is $\gamma=2$, see~\eqref{eq:grading-Q-ex}.  For the convenience of the reader let us state the stability results without additional weights in this special case. 

\begin{corollary}
  \label{cor:stability}
  Let $\mathcal{T}_0$ be a quadrilateral mesh and let~$\mathcal{T}$ be generated from~$\mathcal{T}_0$ by the Q-RG or the Q-RB refinement. 
  Let  $1 \leq p \leq \infty$ and let $q$ be as in~\eqref{eq:final-q}, see also last column in Table~\ref{tab:decay} for the values of~$q$.
  \begin{enumerate}
  \item
    \label{itm:stability-Lp}
    If
    $2^{2\abs{\frac 12 - \frac 1p}} < q^{-1}$,
    then there exists a constant~$c_0 > 0$ such that
    \begin{align*}
      \bignorm{\Pi u}_p &\leq c_0 \bignorm{u}_p \qquad \text{for all $u \in L^p(\Omega)$.}
    \end{align*}
    In particular, the stability holds for $p \in [1,\infty]$ provided that  $K \geq 2$. Furthermore, for P-quadrilateral~$\mathcal{T}_0$ and Q-RG it additionally holds for $p \in (1,\infty)$ and $K=1$. 
  \item
        \label{itm:stability-W1p}
        If $ 2^{1+2\abs{\frac 12 - \frac 1p}} < q^{-1}$, then there exists a constant $c_1 > 0$ such that
    \begin{align*}
      \bignorm{\nabla \Pi u}_p &\leq c_1 \bignorm{\nabla u}_p \qquad \text{for all $u \in W^{1,p}(\Omega)$.}
    \end{align*}
  \item \label{itm:stability-W12}
    If $K \geq 2$, then we have $q < \frac 12$ and there is a constant $c_2 > 0$ such that
    \begin{align*}
      \bignorm{\nabla \Pi u}_2 &\leq c_2 \bignorm{\nabla u}_2 \qquad \text{for all $u \in W^{1,2}(\Omega)$.}
    \end{align*}
  \end{enumerate}
  The constants~$c_0, c_1,c_2$ only depend on~$K$, $\chi(\mathcal{T}_0)$ and the constants in~\eqref{eq:hT}.
\end{corollary}

Corollary~\ref{cor:stability} extends the results by~\cite{AliFunkenSchmidt2022} on $W^{1,2}$-stability for the Q-RG and Q-RB refinement from $K \in \set{2,\dots,9}$ to $K \geq 2$. 
Moreover, we allow for quadrilateral initial meshes, while~\cite{AliFunkenSchmidt2022} is restricted to P-quadrilateral initial meshes. Additionally, we cover also the case of $L^p$- and $W^{1,p}$-stability.

\begin{remark}[Zero boundary values]
  \label{rem:zero}
Both Theorem~\ref{thm:stability} and Corollary~\ref{cor:stability} generalise to spaces with zero traces on (a subset of) $\partial \Omega$. 
 We refer to~\cite[Sec.~3.6]{DieningStornTscherpel2021} for details on the modification of the arguments. 
 The key ingredient is that the decomposition operators respect zero boundary values and a function space decomposition similar to~\eqref{eq:space-decomposition} is available. 
 Also quasi-interpolation operators preserving zero boundary values are available,  see~\cite{BernardiGirault1998}.
\end{remark}

\subsubsection*{Acknowledgements}
The work by T.T. was supported by the
German Research Foundation (DFG) via grant TRR 154, subproject C09, project number 239904186, and by the Graduate School CE within Computational
Engineering at Technische Universität Darmstadt.

\printbibliography
\end{document}